\documentclass[a4paper,reqno]{amsart}
\usepackage{amssymb,amscd}

   \renewcommand{\a}{\alpha}

  \newcommand{\w}{\omega}

   \newcommand{\R}{\mathbb{R}}
   \newcommand{\C}{\mathbb{C}}
  \newcommand{\Z}{\mathbb{Z}}

   \newcommand{\cF}{\mathcal{F}}
   
   \newcommand{\cJ}{\mathcal{J}}

\newtheorem{proposition}{{\bf Proposition}}
\newtheorem{lemma}{{\bf Lemma}}
\newtheorem{theorem}{{\bf Theorem}}
\newtheorem{corollary}{{\bf Corollary}}

\newtheorem{remark}{Remark}

\begin{document}

\author[J.\ C.\ Marrero]{Juan C.\ Marrero}
\address{Juan C.\ Marrero:
ULL-CSIC Geometr\'{\i}a Diferencial y Mec\'anica Geom\'etrica\\
Departamento de Matem\'atica Fundamental, Facultad de
Ma\-te\-m\'a\-ti\-cas, Universidad de la Laguna, La Laguna,
Tenerife, Canary Islands, Spain} \email{jcmarrer@ull.es}

\author[D.\ Mart{\'\i}nez  Torres]{David  Mart{\'\i}nez  Torres}
\address{David Mart{\'\i}nez: Centro de An\'{a}lise Matem\'{a}tica, Geometria e Sistemas Din\^{a}micos,
Departamento de Matem\'atica, Instituto Superior T\'ecnico, Av. Rovisco Pais, 1049-001
Lisboa, Portugal}
 \email{martinez@math.ist.utl.pt}

 \author[E. \ Padr\'on]{Edith Padr\'on}
\address{Edith Padr\'on:
ULL-CSIC Geometr\'{\i}a Diferencial y Mec\'anica Geom\'etrica\\
Departamento de Matem\'atica Fundamental, Facultad de
Ma\-te\-m\'a\-ti\-cas, Universidad de la Laguna, La Laguna,
Tenerife, Canary Islands, Spain} \email{mepadron@ull.es}

\thanks{The first and third authors have been
partially supported by MEC (Spain) grants  MTM2009-13383,
MTM2009-08166-E, and the project of the Canary government
SOLSUBC200801000238. The second author has been partially supported by
 the Funda\c{c}\~{a}o para a Ci\^{e}ncia e a Tecnologia (FCT /
Portugal). The authors  would like to thank Prof. 
Liviu Ornea and M. Verbitsky for helpful comments  related to locally conformal symplectic geometry, and the referee 
for his/her very valuable comments.}


\title[{Universal models for locally conformal symplectic structures}]
{Universal models via embedding and reduction for locally conformal symplectic structures}

 \begin{abstract}
We obtain universal models for several types of locally conformal
symplectic manifolds via pullback or reduction. The relation with
recent embedding results for locally conformal K\"ahler manifolds
is discussed.

\end{abstract}
\maketitle

\section{Introduction and statement of results}

A  manifold $M$ endowed with a nondegenerate $2$-form $\Phi$  is
an  almost symplectic manifold. An almost symplectic manifold
$(M,\Phi)$ is said to be {\it locally conformal symplectic
(l.c.s.)} if for each $x\in M$, there exist an open neighborhood
$U$ of $x$ and a function $\sigma:U\to \R$ such that
$(U,e^{-\sigma}\Phi)$ is a symplectic manifold, i.e.,
$d(e^{-\sigma}\Phi)=0$ (\cite{GL,Va85}). This type of manifolds are included in the category of Jacobi manifolds.
 In fact,  the leaves of the characteristic foliation of a Jacobi manifold are contact or l.c.s. manifolds
 (see, for instance \cite{DLM, GL, Ki}). For manifolds of
dimension greater than 2, an assumption we make from now on, the l.c.s. condition is equivalent to 
\begin{equation}\label{eq:twclosed}
d\Phi=\w\wedge\Phi,
 \end{equation}
 where $\w$ is a closed 1-form, the {\it Lee  $1$-form}. The 2-form $\Phi$ is referred to  as a {\it l.c.s. form}. 
Recalling that any closed 1-form defines a twisted de Rham cohomology, equation
\ref{eq:twclosed} describes a l.c.s. form as a non-degenerate 2-form which is closed in a
twisted de Rham cohomology complex.  This viewpoint is relevant to draw analogies with symplectic geometry.

If $\Phi$ is a l.c.s. form, then so is $f\Phi$ for any $f\in C^\infty(M)$ no-where vanishing.
The l.c.s. forms $\Phi$ and $f\Phi$ are said to belong to the same {\it conformal class}. We will always assume $f$ to be 
positive, so our conformal classes will be -strictly speaking- positive conformal classes.

 A salient feature of l.c.s. structures is that they provide a framework for Hamiltonian mechanics more general than
  the one provided by symplectic structures (see \cite{Va85} or for instance, the recent paper by Marle \cite{Ma} 
where the theory of conformally Hamiltonian vector fields was applied to the Kepler problem).

 It is  natural to investigate up to which extent properties of symplectic manifolds and techniques in symplectic
 geometry generalize to l.c.s. geometry. In the symplectic context, for instance, 
 there is a noteworthy work on embeddings (see, for example, \cite{Ti}) and on reduction 
(see the book by Ortega and Ratiu \cite{OR} and references therein; see also the book by 
Marsden et al \cite{MMOPR} for Hamiltonian reduction by stages).  
 
For l.c.s. geometry, some results on the group of automorphisms
 of a l.c.s. structure \cite{HK99}, on reduction \cite{HK01}, on Moser stability type results \cite{BK09} and on 
existence of l.c.s. structures on open 
manifolds via a h-principle \cite{FF10} have been obtained.
 Another very active line of research in the subject is centered in  {\it locally conformal K\"ahler (l.c.K.)} manifolds.
 These are complex manifolds with a Hermitian metric locally conformal to a K\"ahler one; the underlying l.c.s. structure
 is defined by the 2-form associated to the Hermitian metric. The role played by l.c.K manifolds within l.c.s. manifolds
 is analogous to the one of K\"ahler  manifolds within symplectic ones.  Among the very remarkable recent results in l.c.K.
 geometry, one finds an analog of Kodaira embedding theorem for a subclass of l.c.K manifolds \cite{OV10}.

Kodaira embedding theorem is a very good example of a result in
K\"ahler geometry, which with the appropriate formulation holds
also in symplectic  geometry. Namely, Tischler \cite{Ti} proves
that any integral, compact, symplectic manifold symplectically
embeds in some projective space. In other words, projective spaces
with the integral Fubini-Study symplectic form are universal
models for integral symplectic structures in compact manifolds.
With regard to in which sense Tischler embedding relates to
Kodaira's result, it is known that in general one cannot find
holomorphic embeddings which at the same time pull back the
Fubini-Study metric to a (suitable multiple) of the given K\"ahler
metric. But one easily goes from the holomorphic to the symplectic
embedding by applying Moser stability to  the convex combination
of the two cohomologous K\"ahler forms.

Motivated by the aforementioned results of Tischler, and Ornea and Verbitsky, in this paper we take up the problem
of investigating the existence of compact universal models for l.c.s. structures. Roughly, this amounts to finding families
of compact l.c.s. manifolds -which will be rather special- together with a procedure -either pullback or reduction
 (though for the latter compactness will be dropped)- which allows
us to produce any given l.c.s. structure under reasonable constraints.

Our first result provides a positive answer for  a type of l.c.s.
structures, exact l.c.s. structures with integral period lattice
on compact manifolds (see sections \ref{sec:tw} and  \ref{sec:lcs} for background on l.c.s. structures).

 \begin{theorem}\label{thm:emblcs} Let $(M,\Phi,\a)$ be a compact manifold of dimension $2n$ endowed
 with an exact l.c.s. structure, whose Lee form $\omega$ has integral period lattice. Then, for any $N\geq 4n+2,$ there exist
 an embedding
$\Psi\colon M\to S^{2N-1}\times S^1$
and a real number $c, c>0, $ such that
\begin{equation}\label{embedlcs}
\Psi^*(c\eta_N)=\alpha,\;\;\;\; \Psi^*(d\theta)=\omega,\;\;\;\;
\Psi^*(c\Phi_N)=\Phi,
\end{equation}
where $\eta_N$ is the standard contact $1$-form on $S^{2N-1}$,
$d\theta$ the standard integral 1-form on the circle and $\Phi_N$
the associated standard l.c.s. structure with integral period
lattice on  $ S^{2N-1}\times S^1$.

Using the language introduced in section \ref{sec:tw}, the embedding $\Psi$ is a full strict morphism
into $(S^{2N-1}\times S^1,d\theta)$ which pulls back the homothety class of $\eta_N$ into the homothety
class of $\alpha$ (and thus does the same for the l.c.s. forms).
\end{theorem}

As a consequence of Theorem \ref{thm:emblcs}, we deduce the following result:
\begin{corollary}\label{cor}
Let $M$ be a compact manifold of dimension $2n$ endowed with a l.c.s. structure, whose Lee 
form $\omega$ is not zero in some point of $M$, it has integral period lattice and it is parallel 
with respect to a Riemannian metric on $M.$ Then, for any $N\geq 4n+2,$ there exist
 an embedding
$\Psi\colon M\to S^{2N-1}\times S^1$
and a real number $c, c>0, $ such that
$$
\Psi^*(c\eta_N)=\alpha,\;\;\;\; \Psi^*(d\theta)=\omega,\;\;\;\;
\Psi^*(c\Phi_N)=\Phi,
$$
where $\eta_N$ is the standard contact $1$-form on $S^{2N-1}$,
$d\theta$ the standard integral 1-form on the circle and $\Phi_N$
the associated standard l.c.s. structure with integral period
lattice on  $ S^{2N-1}\times S^1$.
\end{corollary}

The manifold $S^{2N-1}\times S^1$ admits many l.c.K. structures with integral period
lattice associated to diffeomorphisms with linear Hopf manifolds $H_A$ \cite{KO05}
(see section \ref{sec:lck} for background on l.c.K. structures). The standard l.c.s.
form with integral period lattice in  theorem \ref{thm:emblcs} underlies the l.c.K.
form associated to obvious diffeomorphisms to several diagonal Hopf manifolds. In
\cite{OV09,OV10} it is shown that any compact  l.c.K. manifold of complex dimension
at least 3 with automorphic potential -an appropriate generalization of  exact l.c.s. structures in the l.c.K. setting for which the underlying l.c.s. structure is exact-  admits a holomorphic embedding into a linear Hopf manifold $H_A$.

Our second result asserts that the relation between theorem
\ref{thm:emblcs} and Ornea and Verbitsky embedding, mimics the
relation between Tischler and Kodaira embeddings.

\begin{theorem}\label{thm:mosercor} Let $(J,\Phi_g,r)$ be a l.c.K.
structure with automorphic potential and integral period lattice on a
compact manifold $M$. Let $(M,J,\Phi_{g'},r')$ be the l.c.K. structure with
automorphic potential and integral period lattice induced by any of the holomorphic
 embeddings  $\Psi\colon (M,J)\rightarrow (H_A,J_A)$ in \cite{OV09,OV10}, where
 $(H_A,J_A)$ is endowed with a  l.c.K. structure $\Phi_A$ with integral period
 lattice as described in \cite{KO05,OV10a}. Then there exist  diffeomorphisms
 $\varphi\colon M\rightarrow M$ and  $\phi\colon H_A\rightarrow S^{2N-1}\times S^1$, such that
\begin{itemize}
 \item $\phi$ pulls back the standard l.c.s. form $\Phi_N$ on the sphere $S^{2N-1}\times S^1$  to the positive conformal
 class of $\Phi_A$.
\item $\varphi$ is isotopic to the identity and pulls back the l.c.s. form $\Phi_{g'}$ to the positive
 conformal class of $\Phi_g$.
\end{itemize}

Therefore,
\[(\phi\circ \Psi\circ \varphi)^*\Phi_N=f\Phi_g, \]
where $f$ is a strictly positive function. Equivalently, $\phi\circ \Psi\circ \varphi$ is a full morphism
 which pulls back the conformal twisted cohomology class of $\Phi_N$ into the 
conformal twisted cohomology class of $\Phi_g$.

The diffeomorphisms $\varphi$ and $\phi$ are constructed via the Moser stability result in \cite{BK09}.
\end{theorem}

Theorem \ref{thm:emblcs} provides a way of producing all exact l.c.s. structures with integral
period lattice on compact manifolds via pullback (or restriction). Very much as in symplectic
geometry, one can  give conditions so that a reduction process is possible for l.c.s. structures
 \cite{HK01}. Thus one may ask about the existence of universal models for l.c.s. structures via
 reduction. Our third main theorem gives a positive answer to this question for l.c.s. structures
  of the first kind on manifolds of finite type, and it is a natural generalization of results in \cite{GT89,MT97,LMT98}.

\begin{theorem}\label{thm:red} Let $(M,\Phi,\a)$ be a finite type manifold of dimension $2n$,
endowed with a l.c.s. structure of the first kind with rank $k$ period lattice $\Lambda$.
 Then, for any $N\geq 4n+k,$ the l.c.s. manifold $(M,\Phi,\a)$ is isomorphic to the l.c.s.
reduction of certain strongly reducible submanifold of \[(\R\times  \cJ^1(\mathbb{T}^k\times \R^N),
\Phi_{N,\Lambda},\alpha_{N,k},\w_{\Lambda}).\]
 The l.c.s. structure $\Phi_{N,\Lambda}$ is of the first kind with potential 1-form $\alpha_{N,k}$ the
 canonical 1-form in the first jet space of $\mathbb{T}^k\times \R^N$. Its Lee form $\w_{\Lambda}$ has period lattice $\Lambda$.
\end{theorem}

We also prove an equivariant version of the previous theorem. 

\begin{theorem}\label{thm:redeq}
Let $G$ be a compact connected Lie group which acts on a finite type l.c.s. manifold $(M,\Phi,\alpha)$ 
of the first kind with rank $k$ period lattice $\Lambda$, via a l.c.s. action $\psi:G\times M\to M$ of the first kind.
 Then, for a sufficiently large integer $N,$ $(M,\Phi,\alpha,\psi)$ is isomorphic to the l.c.s. equivariant
 reduction by  a certain $G$-invariant strongly reducible submanifold of the l.c.s. structure of the
 first kind $(M_{k,N}={\Bbb R}\times {\mathcal J}^1({\Bbb T}^k\times {\Bbb R}^N), \Phi_{N,\Lambda},\alpha_{k,N},
\omega_\Lambda,\psi_{k,N})$, where  $\psi_{k,N}:G\times M_{k,N}\to M_{k,N}$ is a l.c.s. action of the first kind.  
\end{theorem}

In looking at the problem of existence of universal l.c.s. manifolds linking
with the results in \cite{OV09,OV10}, one is naturally led to ask about the existence of
 (compact) universal models for compact manifolds endowed with an arbitrary 1-form. This is
 a problem that was addressed in much more generality in \cite{NR61}, where universal models
 for (principal) connections on principal bundles for  compact groups were constructed. For
 $U(1)$ the universal models are $S^{2N-1}\rightarrow \mathbb{C}\mathbb{P}^{N-1}$  with the
 standard contact 1-form $\eta_N$. Any 1-form on a manifold $M$ defines a connection on the
 trivial principal bundle $M\times U(1)$. By  \cite{NR61} one produces a bundle morphism
 $M^n\times U(1)\rightarrow S^{8n+3}$, which composed by the right with the inclusion
 $M\rightarrow M\times \{1\}$ pulls back the standard contact 1-form to the given 1-form.
 It turns out that if one is interested not in every compact group but just in $U(1)$, a
 slightly different proof allows to cut down substantially the dimension of the target sphere from $8n+3$ to $4n+3.$

 \begin{theorem}\label{t1}
 Let $M$ be a compact manifold of dimension $n$ and $\Theta$ be a $1$-form on $M$.
 Then for any $N\geq 2n+2,$ there exist an embedding $\Psi:M\to S^{2N-1}$ and a real number $c, c>0, $ such that
 $$\Psi^*(c\eta_{N})=\Theta.$$

In particular, if $\Theta$ is a contact 1-form one obtains an
strict contact embedding between the contact manifold $(M,\Theta)$
and $(S^{2N-1},c\eta_N).$
 \end{theorem}

The paper is organized as follows. In section \ref{sec:1} we will show that a
 universal model (via embeddings) of a compact manifold endowed with a $1$-form  is
 the $(2N-1)$-sphere with its standard contact structure (up to the multiplication by a
 constant). In section \ref{sec:tw} we will recall some aspects of twisted de Rham complexes and
 their cohomology; this setting allows to introduce l.c.s. structures as a twisted version of symplectic structures.
 In section \ref{sec:lcs} will prove that for a compact exact l.c.s. manifold $M$ with  integral period lattice,
there exist a natural number $N$ and  an embedding which  pulls back the standard l.c.s. structure with integral
 period lattice in
$S^{2N-1}\times S^1$ to the l.c.s structure on $M$.
In the particular case of l.c.K. manifolds, we will relate our results with the ones proved recently by Ornea and Verbitsky.
 In section \ref{sec:4},  we will describe
 a
 universal model for reduction of a l.c.s. manifold  of the first kind (theorem \ref{thm:red}). An equivariant
 version of this last result is proved in section \ref{sec:5} (see theorem \ref{thm:redeq}).  The paper ends 
with our conclusions,  a description of future research directions and an appendix where we show the non-exactness of the Oeljeklaus-Toma l.c.K. structures. 

\section{Universal models for 1-forms}\label{sec:1}

In this section we show that the spheres with their standard contact structures are
compact universal models for compact manifolds endowed with 1-forms. An upper bound for the dimension  
of the corresponding model sphere in terms of the dimension of the given manifold is also obtained.

Given a manifold endowed with a 1-form $(M,\Theta)$, it is always possible to
induce the 1-form via an embedding in some Euclidean space endowed with a linear
1-form. Specifically, in $\R^{2n}=T^*\R^n$ with coordinates $x_1,y_1,\dots,x_n,y_n$,
we consider the Liouville 1-form
\[\lambda_n=\sum_{j=1}^n y_jdx_j.\]
A manifold $M$ can always be embedded as a closed submanifold of some
Euclidean space $\R^N$, and  $\Theta$ can be assumed to be the restriction
 of $\bar{\Theta}\in \Omega^1(\R^N)$. Using the universal property of the
 Liouville 1-form in the cotangent bundle, the restriction of $\bar{\Theta}
 \colon \R^N\rightarrow T^*\R^N$ is shown to provide an embedding with the desired property.

If our manifold is compact, we would like to have a similar result but with compact universal models
 as well. Work of Narasimhan and Ramanan \cite{NR61} shows that a solution is given by $(S^{2n-1},\eta_n)$,
 where the standard (contact) 1-form $\eta_n$ is the restriction to the sphere of the 1-form
\[\eta_n=\frac{1}{2}\sum_{j=1}^n(y_jdx_j-x_jdy_j).\] Their result
fits into the more general framework of existence of universal
connections
 for principal bundles for compact groups. More precisely, they give a common construction
  for all unitary groups which includes a bound in the dimension of the target sphere.
  If one is just interested in $U(1)$, it is possible to find an approach which allows
  to obtain target spheres of smaller dimension than in \cite{NR61}.

\begin{proof}[Proof of theorem \ref{t1}] Firstly, 
Whitney's Theorem grants the existence of an embedding $i\colon M\to \R^{2n}$.
We let $U$ be a neighborhood of  $i(M)$ such that its closure $\bar{U}$ is compact.
Denote by  $\bar\Theta$ an extension of $\Theta$ to $\bar{U}.$ Then,
\begin{equation}\label{desco}
\bar\Theta=\sum_{i=1}^{p} f_idx_i, \mbox{ with }p\leq 2n,
\end{equation}
 where $(x_1,\dots, x_{2n})$ are the restriction to $\bar{U}$ of the standard coordinates
  in $\R^{2n}$,  and $f_i\in C^\infty(\bar{U})$.

  Since $\bar{U}$ is compact, there exists $r_1>0$ such that
 $$\sum_{k=1}^{p} ((f_k(x))^2 + (x_k)^2)< r_1^2, \;\;\forall x=(x_1,\dots, x_{2n})\in U.$$

 Now, we consider the map  $\Psi_1:U\to {\mathbb R}^{2p +2}$ given by
 $$\Psi_1(x)=(x_1,f_1(x),\dots ,x_p, f_p(x),\sqrt{r_1^2-\sum_{k=1}^p((f_k(x))^2+ (x_k)^2)},0)$$
 which satisfies that  $\Psi_1(M)\subseteq S^{2p+1}(r_1)$ and
$\Psi_1^*(\eta_{p+1})=\bar\Theta -d\varphi,$ where $S^{2p+1}(r_1)$
is the sphere of dimension $2p+1$ and radius $r_1$, and $\varphi$
is the function on $U$ given by
 $$\varphi=\frac{1}{2}(\sum_{k=1}^pf_kx_k).$$

 Using again that $\bar{U}$ is compact, we deduce that there exists $r_2>0$ such that
 $$\gamma(x)=1 + (\varphi(x))^2 + \sum_{k=1}^p((f_k(x))^2 + (x_k)^2)< r_2^2, \mbox{ for all }x\in U.$$

 Then, the function $\Psi_2\colon U\to {\mathbb R}^{2p+4}$ defined by
$$\Psi_2(x)=(x_1,f_1(x),\dots,x_p, f_p(x), \sqrt{r_2^2-\gamma(x)},0,\varphi(x),{1})$$
induces an embedding  $\Psi_2:U\to S^{2p+3}(r_2)$ such that
$$\Psi_2^*(\eta_{p+2})=\bar\Theta.$$

Finally, if we consider the homothety $\Psi_3:S^{2p+3}(r_2)\to
S^{2p+3}=S^{2p+3}(1)$ given by
$$\Psi_3(x)=\frac{x}{r_2},$$
we have that $\Psi_3^*(\eta_{p+2})=\frac{1}{r_2^2}\eta_{p+2}.$
This ends the proof of our result.
\end{proof}

\begin{remark}\label{r2}
{\rm There is a clear analogy between the proof of theorem \ref{t1} and Tischler embedding theorem: as a first
 step one obtains a map into the sphere (resp. projective space) using basically the universal property of 
cotangent bundles (resp. that $\mathbb{C}\mathbb{P}^\infty$ is the Eilenberg-MacLane space $K(\Z,2)$). That
 map gives a solution up to an exact 1-form (resp. 2-form). Then one needs to use special properties of 
the standard 1-forms $\eta_N$ (resp. the Fubini-Study
2-forms) which makes a correction possible at the expense of increasing by two the dimension of the target.
}
\end{remark}

\begin{remark}{\rm  Our proof is similar to the lemma in  \cite{NR61} section 3, which allows
to obtain universal models for principal connections on trivial $U(n)$-bundles over (subsets of)
 Euclidean space. The difference is that what we make in two steps (firstly getting the result up
  to an exact form and then finding a suitable correction) in  \cite{NR61} is done in just one step
   and for all unitary groups. It is that what allows to cut down the dimension from $8n+3$ to $m,$ with $m\leq 4n+3.$}
\end{remark}

 
\section{Twisted de Rham complexes and local conformal closedness}\label{sec:tw}

In this section we recall a few facts about twisted de Rham differentials  and their cohomology, which will be useful for
our understanding of l.c.s. structures. 

Let $M$ be a manifold. The vector space of smooth functions acts on $\Omega^*(M)$ by $C^\infty(M)$-automorphisms
\[\Omega^*(M)\overset{\mathrm{e}^f}{\rightarrow} \Omega^*(M),\, f\in C^\infty(M).\]
This is not a chain map but it becomes so if we consider the complexes 
\begin{equation}\label{eq:inner}
 \Omega^*(M,d)\overset{\mathrm{e}^f}{\rightarrow}\Omega^*(M,d_{df}),\, f\in C^\infty(M),
\end{equation}
where we use the twisted de Rham differential
\begin{equation}\label{eq:exactdr}
d_{df}(\alpha):=d\alpha-df\wedge \alpha.
 \end{equation}
Any  1-form $\w$ can be used to twist the de Rham differential into $d_\w$ as in (\ref{eq:exactdr}) 
substituting $df$ by $\omega$. In this case $d_\w^2=0$ if and only
if $\w$ is closed. Generalizing (\ref{eq:inner}),  smooth functions act on twisted de Rham complexes
 \begin{equation}\label{eq:innerw}
 \Omega^*(M,d_\w)\overset{\mathrm{e}^f}{\rightarrow}\Omega^*(M,d_{\w+df}),\, f\in C^\infty(M),
\end{equation}
and the isotropy of any twisted de Rham complex is determined by  the constant functions.
We call the equivalence classes
 \emph{conformal classes of twisted de Rham complexes}; we speak of  \emph{homothety classes of twisted de Rham complexes}
if we just consider the action of constant functions. Clearly, twisted de Rham complexes are 
in bijection with closed 1-forms; the action of functions described above corresponds to the action given by adding the
differential of the function, and conformal classes of twisted de Rham complexes correspond to cohomology classes of 1-forms. 
In particular  the conformal class of de Rham complex  corresponds to exact 1-forms.

 If $\omega$ is closed, the cohomology of the complex $\Omega^*(M,d_\w)$ is the \emph{twisted de Rham cohomology} 
$H_{\w}^*(M)$ (also referred to
in the literature as \emph{Lichnerowicz cohomology} or \emph{Morse-Novikov cohomology}), and (\ref{eq:innerw}) induces isomorphisms  
of twisted cohomologies. The twisted de Rham cohomology of a conformal class of twisted de Rham complexes
 is the twisted de Rham cohomology
of any of its representatives.  The homotethy class of the twisted de Rham cohomology of a twisted complex is 
its twisted de Rham cohomology modulo automorphisms induced by the constants.

Let $\w$ be a closed 1-form in $M$. Then $\omega$ can be identified with the additive character
\[\w\colon H_1(M,\Z)\rightarrow  \R.\]
The image of $H_1(M,\Z)$ (or $\pi_1(M)$) by $\w$ is a lattice $\Lambda$ inside of $\R$. 
We define the \emph{period lattice and rank} of $(M,\w)$ to be $\Lambda$ and its rank, respectively.
 In particular a discrete period lattice
is the same as a rank 1 period lattice. We will say that $\w$ has \emph{integral period lattice} if $\Lambda=\Z\subset \R$.
These are invariants of the conformal classes of twisted de Rham complexes.

 Let $(M,\w)$ and $(M',\w')$ be manifolds endowed with closed 1-forms. A smooth map 
 $\phi\colon (M,\w)\rightarrow (M',\w')$ is a \emph{morphism} if it pullbacks the conformal class 
of $\Omega^*(M,d_{\w'})$ into the conformal class of $\Omega^*(M,d_{\w})$. The morphism is \emph{strict} 
if it maps one twisted complex into the other. 
Alternatively,  $\phi$ is a morphism if $[\phi^*\w']=[\w]\in H^1_{\mathrm{dR}}(M)$, and it is strict
if the equality occurs at the level of 1-forms. If $\phi^*\w'=\w+df$ we call $f$  a \emph{scaling function}
 (which is unique up to
constants).

 For a morphism $\phi\colon (M,\w)\rightarrow (M',\w')$ we have $\Lambda\subset \Lambda'$, and
 thus it is rank decreasing. A morphism
is called \emph{full} if $\Lambda= \Lambda'$.

Given a morphism  $\phi\colon (M,\w)\rightarrow (M',\w')$ and $f$ a scaling function, there is an induced homomorphism
\begin{eqnarray}\label{eq:twistedcoh}\nonumber
 \phi^*\colon H^*_{\w'}(M')&\longrightarrow & H^*_{\w}(M)\\
   \beta' &\longmapsto &\mathrm{e}^{-f}\phi^*\beta'.
\end{eqnarray}
To get rid of the choice of scaling function one has to pass to the homothety class of the twisted de Rham
complexes.

\subsection{Twisted de Rham cohomology and de Rham cohomology}

There are two natural ways in which twisted de Rham cohomology can be related to de Rham cohomology. They correspond 
to ways of neglecting the non-exactness of $\w$: working locally or going to a suitable covering space.  

\subsubsection{Local conformal closedness.} Recall that a form $\beta\in \Omega^k(M)$ is  said to be
\emph{ locally conformally closed } if for each $x\in M$, there exist an open neighborhood $U$
of $x$ and a function $\sigma:U\to \R$ such that $e^{-\sigma}\beta$ is closed. 
\begin{remark}\label{rem:isotropic}{\rm Depending 
on the local behaviour of $\beta$ there might be no uniqueness up to additive constant in the choice of $\sigma$. One way to attain
such uniqueness is to ask $\beta$ at each point not to have isotropic hyperplanes. }
\end{remark} 

Let $U_i$, $i\in I$, be an open cover so that $\w_{\mid U_i}$ is exact. Then 
the inclusion $(U_i,0)\hookrightarrow (M,\w)$ is a morphism. If $\beta$
is $d_\w$-closed, then by (\ref{eq:twistedcoh}) 
$\mathrm{e}^{-f_i}\beta$ is closed in $U_i$, where $f_i$ is a scaling function.
 In particular $\beta_{\mid U_i}$ is locally conformally closed.

Conversely let $\beta$ be a locally conformally closed form such that the local functions 
$\sigma_i\colon U_i\rightarrow \R$, $i\in I$,
are unique up to constant. Then the Cech cocycle
$\beta_{\mid U_i}\in H^*_{d\sigma_i}(M)$ glues into a cocycle $\beta\in H^*_{\w}(M)$, where $\w_{\mid U_i}=d\sigma_i$. 

\subsubsection{Covering spaces and automorphic forms.}
Let $\w\in \Omega^1(M)$ be closed.  A covering space $\pi\colon\tilde{M}\rightarrow (M,\w)$ is called exact if
  $\tilde{\w}:=\pi^*\w$ is exact. This is equivalent to saying that 
\[\pi\colon (\tilde{M},0)\rightarrow (M,\w)\] is a morphism.  Therefore according to (\ref{eq:twistedcoh}), 
 $\mathrm{e}^{-f}\pi^*\beta$ is closed whenever $\beta\in \Omega^k(M)$ is $d_\w$-closed, where $f$ is a scaling function.
The \emph{smallest exact covering space} of $(M,\w)$ is the one with fundamental group the kernel of the additive character $\w$.
 
 Consider the multiplicative character
\begin{eqnarray}\label{eq:character}\nonumber
 \chi \colon \Gamma &\longrightarrow &\R^{>0}\\
      \gamma &\longmapsto & \chi(\gamma),\,\gamma^*\mathrm{e}^f=\chi(\gamma)\mathrm{e}^f.
\end{eqnarray}
For every $\beta\in \Omega^*(M)$ the group of deck transformations $\Gamma$ acts on  $\mathrm{e}^{-f}\pi^*\beta$ by homotheties 
$\gamma\rightarrow \chi(\gamma)$. We denote the subcomplex of all forms with that property by $\Omega^*(\tilde{M})^{\chi}$,
and we refer to them as automorphic forms (w.r.t. $\chi$). Note that $\chi$ 
is related with  the additive character $\w$ in a straightforward manner: the additive character induces
 additive character on $\tilde{M}$ by pull back or equivalently by using 
\begin{equation}\label{eq:exact}
1\rightarrow \pi_1(\tilde{M})\rightarrow \pi_1(M)\rightarrow \Gamma\rightarrow 1,
\end{equation}
that we still denote by the same name, and one has
\begin{equation}\label{eq:rankchar}
 -\mathrm{ln}\chi=\w.
\end{equation}
In particular this allows to read the additive character by data in the exact covering space (see also \cite{PV}).

Conversely, we say that  $\tilde{\beta}\in \Omega^*(\tilde{M})$ is \emph{automorphic} if the group of deck transformations acts
by homotheties on $\tilde{\beta}$.  We denote by  $\chi_{\tilde{\beta}}$ the corresponding character.  Using (\ref{eq:exact}), 
the character induces a character in $M$, and taking minus its logarithm an additive one, that is an element in $H^1(M,\Z)$.
Let $\w_{\tilde{\beta}}\in \Omega^1(M)$ be a representative. Then 
\[\pi\colon (\tilde{M},0)\rightarrow (M,\w_{\tilde{\beta}})\]
is a morphism. Let $f$ be a scaling function for $\pi^*\w_{\tilde{\beta}}$. Then by
 (\ref{eq:rankchar}) $\mathrm{e}^{f}\tilde{\beta}$ is invariant
under the action of $\Gamma$, and thus descends to $\beta\in \Omega^k(M)$ which is $d_{\w_{\tilde{\beta}}}$-closed
 if $\tilde{\beta}$ is closed.

We summarize this correspondence in a lemma for its latter use (see also \cite{Ba02}).
\begin{lemma}\label{lem:corresp} Let $\pi\colon\tilde{M}\rightarrow (M,\w)$ be a exact covering space. Then
it determines a character $\chi$ such that 
 the  assignment
\begin{eqnarray}\label{eq:covisom}\nonumber
 \Omega^*(M,d_\w) &\longrightarrow &\Omega^*(\tilde{M},d)^{\chi}\\
   \beta &\longmapsto & \mathrm{e}^{-f}\tilde{\beta},
\end{eqnarray}
where $f$  is a scaling function, is a monomorphism of chain complexes sending forms into automorphic forms. To avoid 
the choice of scaling function one may speak of a monomorphism from the homothety class of  $\Omega^*(M,d_\w)$
to the homothety class of the subcomplex $\Omega^*(\tilde{M},d)^{\chi}$,  which descends  to homothety classes of
twisted de Rham cohomology.

\noindent Conversely, any character $\chi\colon \Gamma\rightarrow \R^{>0}$ determimes a cohomology class  of $H^1(M,\R)$, and for 
any representative $\w_\chi$ a chain map 
\begin{equation}\label{eq:covisom2}
  \Omega^*(\tilde{M},d)^{\chi} \longrightarrow \Omega^*(M,d_{\w_\chi}).
\end{equation}
To get rid of choices one speaks of a well defined map form the homothety class of  $\Omega^*(\tilde{M},d)^{\chi}$
 into the conformal
class  of  $\Omega^*(M,d_{\w_\chi})$. Clearly, both constructions are inverse of each other (when we consider conformal
classes of twisted de Rham complexes in $M$).
\end{lemma}

\subsection{Computations of twisted de Rham cohomology.}  As for computations of twisted de Rham cohomology 
(for $\w$ non-exact), these are hard. If $M$ is connected $H^0_\w(M)=0$, and if additionally  $M$ is compact 
and orientable then $H^{\mathrm{top}}_\w(M)=0$ \cite{Ba02,GL,H}. Under the compactness and orientability assumptions, 
because the twisted differential is a degree zero deformation of the de Rham differential,
 the Euler characteristic of the twisted de Rham complex is the Euler characteristic of $M$. If one further assumes that $\w$
 is parallel
for some Riemannian metric (and non-trivial), then the twisted de Rham complex is acyclic \cite{LeLoMaPa}. 
There are some explicit 
computations by Banyaga describing non-trivial twisted cohomology classes in a particular 4-manifold \cite{BK09,Ba08}.

Our contribution to computations of twisted de Rham cohomology will be showing that the degree 2  (conformal)
twisted de Rham cohomology class associated to the so called 
Oeljeklaus-Toma l.c.K manifolds is non-trivial. We postpone the proof to the appendix \ref{appendix} 
(proposition \ref{pro:otnonexact}), once the necessary material on
l.c.K. structures has been introduced.

\begin{remark}\label{rem:nosurg} 
{\rm It is natural to extend to the twisted setting geometries defined by conditions on forms and their exterior
 differentials. Thus,
in order to get new examples of such structures  one would like to have simple topological constructions 
to produce new twisted cohomology classes
 in a fixed degree. Unfortunately, these seem difficult to come up with (for example,  given $(M,\w)$ and $(M',\w')$, 
it is natural to consider $(M\times M',\w+\w')$; it is true that if $d_\w\beta=
d_{\w'}\beta'=0$, then $d_{\w+\w'}(\beta\wedge \beta')=0$, but the degree is increased).}
\end{remark}


\section{Locally conformal symplectic structures}\label{sec:lcs}

Recall that an almost symplectic manifold $(M,\Phi)$ is said to be
l.c.s. if $\Phi$ is locally conformal closed (\cite{GL,Va85}).  If we are in dimension greater than 2,
 an assumption which we make from now on, isotropic subspaces cannot have codimension 1, 
so a l.c.s. manifold is given
by a closed 1-form $\w$, the Lee form, and a maximally non-degenerate $d_\w$-closed 2-form $\Phi$. In other words,
a l.c.s. form should be understood as a symplectic form in an appropriate twisted de Rham complex. The cohomology class of the
Lee form is the {\it Lee class} of $(M,\Phi)$. The rank and period lattice of the l.c.s. structure are
 the rank and period lattice of its
Lee class.  Several of our results are stated for l.c.s. structures with integral period 
lattice, but they remain valid for discrete lattices.

 Non-degeneracy is clearly a conformal property, and thus it is natural to consider conformal classes of  l.c.s. structures. 
In this respect, it is worth pointing out that in l.c.s. geometry one is often able to get results at the level of conformal classes.
 A good illustration of this fact is the Moser stability result in \cite{BK09} (see theorem \ref{thm:moser}), and the
 reduction by group actions in \cite{HK01}. Of course, it is much desirable to prove statements at the level of homothety 
classes or even of l.c.s.
 forms when possible.

A l.c.s. manifold $(M,\Phi)$ is called \emph{exact} if $\Phi$ is $d_\w$-exact, where $\w$ is the Lee form of $(M,\Phi)$.
 We will use the notation $(\Phi,\a)$ for an exact l.c.s. structure $\Phi$ with fixed potential 1-form $\a$.
 Of course, the information given by either of the tuples $(\Phi,\a), (\Phi,\a,\w)$ is the same, so we often omit the Lee
form.

Very much as in symplectic geometry a l.c.s. form $\Phi$ induces a vector bundle isomorphism $\flat_\Phi:TM\to T^*M,$ given by
\begin{equation}\label{flat-Phi}
\flat_\Phi(v)(x)=i_v\Phi(x)\mbox{ for $x\in M$ and $v\in T_xM.$}
\end{equation}

Now, we consider the Lie algebra  of infinitesimal automorphisms of $(M,\Phi)$, i.e.,
$${\mathfrak X}_\Phi(M)=\{X\in {\mathfrak X}(M)/{\mathcal L}_X\Phi=0\}.$$

Since $\Phi$ is non-degenerate,  we deduce that for all  
 $X\in{\mathfrak X}_\Phi(M)$, ${\mathcal L}_X\omega=0$, i.e., $\omega(X)=constant$. Moreover,
 from $d\omega=0$, we obtain that $\omega([X,Y])=0,$ for all $X,Y\in {\mathfrak X}_\Phi(M).$
Thus, we have  the Lie algebra morphism
$$l:{\mathfrak X}_\Phi(M)\to \R,\;\;\; l(X)=\omega(X)$$
where on $\R$ one takes the commutative Lie algebra structure. In particular the {\it anti Lee vector field } 
\begin{equation}\label{eq:aLee}
E:=-\flat_\Phi^{-1}(\w)
\end{equation}
 is in the kernel of $l$.

If $l\not= 0$ then  we say that $(M,\Phi)$ is a {\it l.c.s. manifold of the first kind}. A choice of 
{\it transverse infinitesimal automorphism (t.i.a.)} $B\in l^{-1}(1)$ produces a 1-form via the formula
\begin{equation}\label{eq:firstk}
\a=-\flat_\Phi(B),
\end{equation} and it can be checked that
\[d_\w\a=\Phi.\]
Therefore l.c.s. structures of the first kind are in particular
exact. We note that being of the first kind -unlike exactness- is
a property of the homotethy class of the l.c.s. structure, but not of the conformal class in general
\cite{Va85} (so in particular being exact is weaker than being of
the first kind). We also remark that $i_{[B,E]}\Phi=0$, which
implies that $[B,E]=0.$

\begin{remark}\label{rem:firstkl} 
{\rm Another way of arriving at the l.c.s. structures of the first kind among exact ones is as follows:
Consider  $(M,\Phi)$ an exact l.c.s. structure and select $\a$ a potential 1-form. In analogy
with symplectic geometry one defines a vector field by  $\a=-\flat_\Phi(B)$ and expects
an special behaviour of $\Phi$ under the flow if $B$. But one gets
\[\mathcal{L}_X\Phi=(1-\w(B))\Phi,\]
and obtains  either a Liouville type condition or a symplectic type condition by imposing 
 $\w(B)=0$ or $\w(B)=1$ (actually $\w(B)\neq 0$,
 but $B$ is rescalled to give 1). We are mainly interested in compact l.c.s. structures, so the Liouville type condition
is impossible since $\Phi^n$ is a volume form. Thus, the symplectic type condition, 
which coincides with being a l.c.s. structure of the first
kind, appears as the relevant subclass of exact l.c.s. structures on compact manifolds.}
\end{remark}

\begin{remark}\label{rem:contact}{\rm L.c.s. structures of the first kind in $M^{2n}$ are  discussed under
the name of contact pairs of type $(n-1,0)$ in \cite{BK10}.}
\end{remark}

An example of l.c.s. manifold of the first kind is $(S^{2n-1}\times
S^1,\Phi_n,\eta_n,d\theta)$,  where the canonical integral 1-form
on the circle $d\theta$ is the Lee form, and the contact $1$-form
$\eta_{n}$ on $S^{2n-1}$  is the potential $1$-form (so
$\Phi_n:=d_{d\theta}\eta_n$). 

The statement of theorem \ref{thm:emblcs} -whose prove we are ready to give-
 is that $(S^{2n-1}\times S^1,\Phi_n,\eta_n,d\theta)$
 are universal manifolds for exact l.c.s. structures with integral period lattice on compact manifolds.

\begin{proof}[Proof of theorem \ref{thm:emblcs}]
 By hypothesis we can write
 $$\Phi=d_\omega\alpha.$$
 Using theorem \ref{t1}, we deduce that for any natural $N\geq 4n+2$ there exist  an embedding  $\Psi_1:M\to S^{2N-1}$
  and a real number
 $c, c>0,$ such that
\begin{equation}\label{eq1}
\Psi_1^*(c\eta_{N})=\alpha.
\end{equation}

From the integrality assumption on the Lee form we conclude the
existence of a smooth map $\tau:M\to S^1$ such that
\begin{equation}\label{eq2}
\tau^*(d\theta)=\omega.
\end{equation}

Now,  the embedding $\Psi: M\to S^{2N-1}\times S^1$ given by
$$\Psi(x)=(\Psi_1(x), \tau(x))$$
satisfies (\ref{embedlcs}),   and this proves theorem \ref{thm:emblcs}.

 Note that  $\Psi\colon (M,\w)\rightarrow (S^1\times S^{2N-1},d\theta)$ is a strict morphism which must be full 
since $\w$ has integral period lattice and 
morphisms are rank decreasing. And by construction the homothety class of $\eta_N$ is pulled back to $\a$.
\end{proof}

\begin{remark}\label{rem:rank1}{\rm  Theorem \ref{thm:emblcs} remains true when the periods of the Lee
form generate the discrete lattice $q\Z$, $q\in \R^{>0}$. One needs to use instead the l.c.s. structures
 of the first kind  $(S^{2N-1}\times S^1,\Phi_{N,q},q\eta_N,d\theta),$ where $\Phi_{N,q}=q\Phi_N.$}
\end{remark}

\begin{remark}\label{rem:firstk}{\rm If $(M,\Phi,\a)$ is a l.c.s. structure of the first kind
with t.i.a. $B$ related to $\a$ as in (\ref{eq:firstk}), then the
Lee form is no-where vanishing. Therefore it defines a foliation
without holonomy, the discreteness of the period lattice being
equivalent to the foliation being a fibration over $S^1$. The
restriction of $\a$ to each leaf is a contact 1-form. So the
l.c.s. structure of the first kind can be understood as a
1-parameter family of exact contact manifolds with a transverse
automorphism (the integration of the t.i.a). With this description
theorem \ref{thm:emblcs} for l.c.s. structures of the first kind
is the appropriate 1-parameter version of theorem \ref{t1} for
contact forms.}
\end{remark}

\begin{proof}[Proof of corollary \ref{cor}]
Under the hypotheses of the corollary, we have that 
$$H_{\omega}^k(M)=\{0\},\,\;\; \mbox{ for all }k$$
(see theorem 4.5 in \cite{LeLoMaPa}). 
Thus, the l.c.s. structure of $M$ is exact and we may apply theorem \ref{thm:emblcs}. 
\end{proof}

\begin{remark}\label{rem:norank} 
{\rm It is natural to define  universal models for  l.c.s. on compact manifold by requiring the existence of embeddings
 into them which are (1)
full (strict) morphisms, and (2) pull back the (strict) conformal class of the l.c.s. form into the given one.
 This would imply the existence of a moduli parametrized by lattices $\Lambda\subset \R$.
As a consequence one would need to have a large supply of compact l.c.s. manifolds with arbitrary period lattice,
 but examples are scarce.
In this respect it is noteworthy the family of Oeljeklaus-Toma l.c.K. structures which have arbitrary 
rank \cite{PV}, and are non-exact as will be shown in appendix \ref{appendix} 
(note that our universal models for reduction are exact l.c.s. manifolds with  arbitrary period lattices,
 but they are non-compact).
As an illustration
of the difficulty of producing examples of l.c.s. structures  consider a contact
 manifold $(N,\eta)$ and the associated exact l.c.s. manifold 
 $(S^1\times N,d_{d\theta}\eta,d\theta)$.
Now let $\Sigma$ be an orientable surface. For the product manifold $S^1\times N\times \Sigma$ finding a 1-form $\alpha$ 
such that $d_{d\theta}\a$ is l.c.s. implies finding contact structures in $\Sigma\times N$. This is a very non-trivial problem
whose solution was only found recently \cite{Bo}. It is natural then to ask wheter given a l.c.s. manifold $(M,\Phi)$, one
can endow $\Sigma\times M$ with a l.c.s. structure (in particular one would be overcoming the problems noted in remark
\ref{rem:nosurg} about producing new closed 2-forms in twisted de Rham complexes out of old ones).}
\end{remark}

\section{Relation with embedding results for l.c.K.
manifolds}\label{sec:lck}
In this section we will discuss the relation between theorem
\ref{thm:emblcs} and recent embedding results by Ornea and
Verbitsky for locally conformal K\"ahler (l.c.K.) structures.

\subsection{Vaisman manifolds}

 A l.c.K. structure on a complex manifold $(M,J)$ is
given by a Hermitian metric $g$ which is locally conformal to a
K\"ahler one. The underlying l.c.s. structure is defined by the
associated 2-form
\[\Phi_g:=g(\cdot,J\cdot).\]
Equivalently, a l.c.K. structure is given by a l.c.s. form $\Phi$ and an integrable compatible almost complex structure $J$.
 If $(\tilde{M},\tilde{J})\rightarrow (M,J)$ is a complex covering space, according to 
lemma \ref{lem:corresp} there is a one to one correspondence between 
 homothety classes of automorphic K\"ahler forms $\Omega$  on $(\tilde{M},\tilde{J})$, and conformal classes of l.c.K.
 structures in $(M,J)$ whose Lee class becomes exact in $\tilde{M}$.

Let $(M,J,\Phi_g)$ be a l.c.K. manifold with Lee form $\w$. The
{\it Lee vector field} $B$ is the metric dual of the Lee form.  By
construction $-JB$ is the anti-Lee vector field  of the underlying l.c.s. as defined  in (\ref{eq:aLee}). A l.c.K.
structure is called {\it Vaisman} if the Lee form is parallel. It
follows that $B,E=-JB$ are both Killing,  preserve $J$, and have
commuting flows ($B-iJE$ is a holomorphic vector field)
\cite{KO05}. If we go to an exact complex covering space $(\tilde{M},\tilde{J})$, and we let $\Omega$ be in the homotethy
class of K\"ahler forms furnished by lemma \ref{lem:corresp}, the lift of the flow of the Lee vector field
is by K\"ahler homotheties. Unlike the case of l.c.K. structures the data $(\tilde{M},\tilde{J},\Omega)$ determines uniquely
the homotethy class of the Vaisman structure $(M,J,\Phi_g)$. Thus, one can take  this second approach
as a definition of Vaisman structure.

By definition the Lee vector field of a Vaisman structure belongs
to $\mathfrak{X}_{\Phi_g}(M)$. Therefore Vaisman structures are
the natural analogs in l.c.K. geometry of l.c.s. structures of the
first kind (see also \cite{BK10}, where compatible almost complex structures are also brought into the picture).
 It should be noted though, that Vaisman structures on
a compact manifold have always rank 1 \cite{OV03}. For simplicity
we will normalize our Vaisman structures so that $l(B)=1$, i.e. we
can take the Lee vector field as t.i.a., and thus
$-\flat_{\Phi_g}^{-1}(B)$ is a potential 1-form for the l.c.s.
structure $\Phi_g$ (and in this way we fix a representative of the homothety class).

Examples of Vaisman manifolds are the diagonal Hopf manifolds $(H_A,J_A,\Phi_A)$ \cite{KO05}.
One rather introduces these Vaisman structures by starting with the covering space
$(\tilde{H}_A,\tilde{J}_A):=(\C^{N}\backslash \{0\},J_{\mathrm{std}})$, and taking
$\Gamma\cong \mathbb{Z}$ generated by the linear action of an invertible matrix $A$ which has
 all its eigenvalues with norm $<1$, and which is diagonalizable. Because of the conditions on its
 eigenvalues $A$ is in the image of the exponential map and has a unique logarithm. That defines a
 1-parameter group of holomorphic transformations whose time 1 map is the action by $A$. In \cite{KO05}, section 3,
 a family $\Omega_{A,q}$, $q\in \R^{>0}$, of K\"ahler forms for which the previous flow acts by K\"ahler homotheties
 is given. The unique normalized Vaisman structure in $(H_A,J_A)$ is denoted by $\Phi_{A,q}$.
 The parameter $q$ is such that the period lattice of the Lee form is $q\Z$.
 Diagonal Hopf manifolds are
 diffeomorphic to $S^{2N-1}\times S^1$ (see the discussion in the proof of proposition \ref{pro:linhopf}).
  The standard l.c.s. structure in theorem \ref{thm:emblcs} is  the one associated to a diagonal Hopf manifold
    where $A=\lambda \mathrm{Id}$ is a real multiple of the identity (for an obvious
 diffeomorphism between $H_{\lambda \mathrm{Id}}$ and $S^{2N-1}\times S^1$).

\subsection{L.c.K structures with automorphic potential}
Theorem \ref{thm:emblcs} holds not just for l.c.s. structures of the first kind with integral period lattice, 
but for exact ones. In \cite{OV09} (see also \cite{OV10})  Ornea and Verbitsky have introduced the notion of l.c.K.
 structure with vanishing Bott-Chern class (or with automorphic potential):  given any closed 1-form $\w$,
 in the presence of a complex
structure the complexified twisted de Rham complex $(\Omega^*(M,\C),d_\w)$  can be split into its holomorphic
 and antiholomorphic components, and so the twisted differential
\[d_\w=\partial_\w+\bar{\partial}_\w.\]
This gives rise to a Bott-Chern cochain complex with cohomology groups $H^{p,q}_{\partial_\w\bar{\partial}_\w}(M)$, 
 (see \cite{OV09} for details); the action of functions in (\ref{eq:innerw}) on twisted de Rham complexes induces
an action on Bott-Chern cochain complexes. Suffice it to say here that for a l.c.K manifold $(M,J,\Phi_g)$ 
 the l.c.K. 2-form defines  the {\it Bott-Chern class}
\[[\Phi_g]\in H^{1,1}_{\partial_\w\bar{\partial}_\w}(M),\]
and that the identity induces a homomorphism
\[H^{1,1}_{\partial_\w\bar{\partial}_\w}(M)\rightarrow H_\w^2(M)\]
sending the Bott-Chern class to the class $[\Phi_g]\in H_\w^2(M)$
(one has the usual equality $i\partial_\w\bar{\partial}_\w=d_\w
d_\w^c$,
 with $d_\w^c=J^*d_\w$).

A l.c.K. structure is said to have {\it vanishing Bott-Chern class}
 if $[\Phi_g]\in H^{1,1}_{\partial_\w\bar{\partial}_\w}(M)$ is trivial. Having vanishing Bott-Chern class
is conformally invariant, and it also implies that the class $[\Phi_g]\in H_\w^2(M)$  is trivial, so the underlying l.c.s.
 structure is exact. It is not clear whether the converse is true or not, this being related
 to the existence of a global  $\partial_\w\bar{\partial}_\w$ lemma \cite{OV09}.

A Vaisman structure has vanishing Bott-Chern class, and a canonical $d_\w d^c_\w$-potential is
given by the constant function 1. Indeed, the equation
\[\partial_\w\bar{\partial}_\w 1=\Phi_g\]
is equivalent to
\[d(J^*\w)=\Phi_g+\w\wedge J^*\w,\]
which in turn is equivalent to
\[\mathcal{L}_B\Phi_g=i_Bd\Phi_g+di_B\Phi_g=0.\]

For a exact complex covering space $(\tilde{M},\tilde{J})$ of $(M,J,\w)$, in the assignment described 
in lemma \ref{lem:corresp}   $d_\w d^c_\w$-potentials for
$\Phi_g$ correspond to automorphic usual $dd^c$-potentials for the automorphic K\"ahler form associated to
 the choice of scaling function.
If $M$ is compact and $\tilde{M}$ is the smallest exact covering space then the $dd^c$-potential is proper \cite{OV10b}.

One advantage of l.c.K. structures with automorphic potential is
that one may construct new ones via small perturbations. For
example going to a exact covering space $\tilde{M}$, one can perturb a bit the complex structure in the base,
lift it, and perturb the initial potential a little bit so that
the K\"ahler condition still holds (and one can also allow for
perturbations of the subgroup $\Gamma$ so that the Lee class can
change).

Examples of l.c.K. structures with automorphic potential are
constructed in the linear Hopf manifolds $(H_A,J_A)$. This is the
same construction as for diagonal Hopf structures, but the
invertible matrix $A$ is just supposed to have eigenvalues of norm
$<1$. The automorphic  potential is constructed by perturbation as
indicated for example in  \cite{OV10a}
 (see also \cite{GO98}): the closure of the orbit of any $A$ as above (by the action by conjugation of the complex general
 linear group) contains diagonalizable matrices $A'$. This implies the existence of a diffeomorphism $H_A\cong H_{A'}$,
 pushing $J_A$ in to a complex structure in $H_{A'}$ that we still denote $J_{A}$. One can assume
 that $J_A$ and $J_{A'}$ are as close as desired.
 Thus the same automorphic potential in the covering space $(\tilde{H}_{A'}\tilde{J}_{A'})$ for 
the Vaisman structure $\Phi_{A'}$ defines an automorphic K\"ahler metric for the
 lift of the complex structure $J_A$. This gives rise to a l.c.K. form with automorphic potential $\Phi_A$ in $(H_{A'},J_A)$ 
(and hence in $(H_{A},J_A)$) by using the fixed
Lee form in $H_{A'}$  and the fixed scaling function in $\tilde{H}_{A'}$.

\begin{remark}\label{rem:family}{\rm
Observe that  for the given diffeomorphism the Lee form is the same for both $\Phi_A$ and $\Phi_{A'}$; also 
the automorphic potential in the covering space is chosen to be the same.
 Note as well  that one can arrange for the existence of $A_t$, $t\in [0,1]$, $A_0=A'$, $A_1=A$ and so
 that the construction holds with parameters (i.e. on has $(H_{A'},J_{A_t},\Phi_{A_t})$, $t\in [0,1]$
 l.c.K. structures with the same automorphic potential (in the covering space) and the same Lee form).}
\end{remark}

The main embedding result of Ornea and Verbitsky is the following:

\begin{theorem}\cite{OV09}\label{thm:holemb} Let $(M,J,\Phi_g,r)$ be a l.c.K. structure with automorphic potential
on a compact manifold of complex dimension at least 3. Then there
exists a holomorphic embedding of $(M,J)$ into a linear Hopf
manifold $(H_A,J_A)$. Moreover, if $(M,J,\Phi_g)$ is Vaisman then
there exist a holomorphic  embedding into a diagonal Hopf
manifold.
\end{theorem}

 Because we want to  eventually prove theorem \ref{thm:mosercor}
 we will work with l.c.K. structures with integral period lattice.

\subsection{Families of l.c.K. structures and Moser type results.}

Extending previous work of Banyaga \cite{Ba02}, Bande and
Kotschick \cite{BK09} have proved a Moser stability type result
for l.c.s. structures. Here we only state a particular case which
will suffice for our purposes.

\begin{theorem} (Corollary 3.3. in \cite{BK09})\label{thm:moser}. Let $\Phi_t$, $t\in [0,1]$, be a smooth family of l.c.s.
structures on a compact manifold $M$ such that the corresponding Lee forms $\w_t$ have the same de Rham cohomology class.
Suppose there exists a smooth family of 1-forms $\a_t$ such that $\Phi_t=d_{\w_t}\a_t$. Then there exists
 an isotopy $\phi_t$ such that $\phi^*_t\Phi_t$ is conformally equivalent to $\Phi_0$ for all $t$.
\end{theorem}

In a compact K\"ahler manifold the convex combination of any two
K\"ahler forms is K\"ahler, so cohomologous K\"ahler forms define
the same symplectic structure up to a diffeomorphism isotopic to
the identity.

The space of l.c.K. structures with fixed Lee form is discussed in \cite{OV09}. One can slightly generalize those results
to obtain  stability for the conformal class of the underlying l.c.s. structures.

\begin{lemma}\label{lem:lckmoser} Let $(M,J)$ be a complex manifold. Then the space of l.c.K. structures with
fixed Lee class is connected. Moreover, the same holds for l.c.K. structures with automorphic potential.
In particular if $M$ is compact the action of the group of diffeomorphisms isotopic to the identity
 on the conformal classes of
 l.c.s. structures with l.c.K. representatives with automorphic potential, has orbits parametrized by the Lee class.
\end{lemma}

\begin{proof} Let $\Phi_g,\Phi_{g'}$ be two l.c.K. structures in $(M,J)$ with the same Lee class $[\w]$.
Let $(\tilde{M},\tilde{J})$ be a exact covering space. Then one has scaling functions
 $f,f'\in C^\infty(\tilde{M})$ and automorphic K\"ahler forms $\Omega=e^{-f}\tilde{\Phi}_g,\Omega'=e^{-f'}\tilde{\Phi}_{g'}$.

The convex combination
$\Omega_t=(1-t)\Omega+t\Omega'$ defines a
family of automorphic  K\"ahler forms (for $\chi(e^{-[\w]})$). Thus, they
define a 1-parameter family of conformal classes of l.c.K.
structures with Lee class $[\w]$. It is easy to find a smooth path
of representatives by just choosing the path of functions $f_t=(1-t)f+tf'$ which have all
additive character $[\w]$.

If the given l.c.K. structures have automorphic potential $r,r'$,
then $r_t=(1-t)r+tr'$ is an automorphic potential for $\Phi_{g_t}$.

Thus,  if we are in a compact manifold we can apply theorem
\ref{thm:moser} and this proves the lemma.
\end{proof}

\begin{remark}\label{rem:exact} {\rm The stability result in lemma \ref{lem:lckmoser} also holds for exact l.c.K.
 structures by applying  Hodge theory  to find potential 1-forms
 \cite{BK09}.}
\end{remark}

All linear Hopf manifolds $H_A$ are diffeomorphic to $S^{2N-1}\times S^1$. We want to show that,
in an appropriate sense, the  conformal class of the  l.c.s. structure induced by any $\Phi_A$
with integral period lattice is unique (see remark \ref{rem:isot}).

\begin{proposition}\label{pro:linhopf} If  $A\in GL(N,\C)$ has eigenvalues of norm smaller than 1, and  $(H_A,J_A,\Phi_A,\w_A)$ is a
l.c.K. structure with automorphic potential and integral period
lattice as constructed in \cite{KO05,OV10a},  then there exists a
diffeomorphism
\[\phi_A\colon (S^{2N-1}\times S^1,d\theta)\rightarrow (H_A,\w_A)\] which is a (full) morphism and pulls back the
conformal class of $\Phi_A$ to the  conformal class of the standard l.c.s. structure with integral period lattice $\Phi_N$.
\end{proposition}
\begin{proof}

According to theorem \ref{thm:moser} we just need to find a diffeomorphism  such that $\phi_A^*\Phi_A$ and $\Phi_N$
can be joined by a (piecewise) smooth path of l.c.K. structures -for possibly different complex structures- 
with integral period lattice, and with potential 1-forms
 varying smoothly.

We do it in several steps. We assume that $A$ is not diagonalizable. By remark \ref{rem:family} we can find $A'$ diagonalizable
and a diffeomorphism $\phi_1\colon H_{A'}\rightarrow H_{A}$ so that $\Phi_{A'}$ and $\phi_1^*\Phi_{A}$ can be joined by a
 smooth family of l.c.K. structures with automorphic potential and integral period lattice, and such that the Lee form
 and automorphic potential are the same. Therefore, we can assume without loss of generality that $A$ is diagonalizable.
 Note as well that
 the deformation argument in \cite{OV10a} may produce many  l.c.K. structures with automorphic potential and integral period
 lattice. By lemma  \ref{lem:lckmoser} they all belong to the same  conformal class.

The second step amounts to comparing all  l.c.K. Vaisman structures with integral period lattice in a given $GL(N,\C)$ orbit.
 In  \cite{KO05} for $A$ diagonal an automorphic potential is given defining a Vaisman structure with integral period lattice.
 The associated K\"ahler metric in $\C^N\backslash \{0\}$ is invariant by conjugation by $GL(N,\C)$, since it is defined by a
  potential. Given any $A,A'$ in the same orbit, if $A'=L*A$, then $L$ pushes any path of  l.c.K. structures with automorphic
  potential and integral period lattice starting at the Vaisman structure $\Phi_A$, into a path of  l.c.K.
 structures with automorphic potential and integral period lattice  starting at the Vaisman structure $\Phi_{A'}$. 
This implies that we can assume $A$ to be
   diagonal.

The automorphic potentials for the  Vaisman  structures with integral period lattice  in \cite{KO05} in diagonal Hopf manifolds,
 depend smoothly on the eigenvalues (see also \cite{OV09}, section 2.2, for an explicit formula).  Let $A$ be diagonal. 
We assume that all eigenvalues have norm $q$ . We let $A'=A/2q$ and consider the convex combination $A_t=(1-t)A+tA'$ with 
eigenvalues whose norm is $q_t$.

A fundamental domain for $H_{A_t}$ is the closed annulus $\mathbb{A}(q_t,1)$ of Euclidean radii $q_t,1$, and the manifold 
is obtained by applying the same diffeomorphism $S^{2N-1}\rightarrow S^{2N-1}$ for all $t$. Let 
$k_t\colon [q_t,1]\rightarrow [1/2,1]$ be the linear orientation preserving diffeomorphism. Then for the product 
decomposition of the annuli into radial and spherical coordinates,
\[k_t\times \mathrm{Id}\colon  \mathbb{A}(g_t,1)\rightarrow \mathbb{A}(1/2,1),\]
is a diffeomorphism which descends to a diffeomorphism
\[K_t\colon H_{A_t}\rightarrow H_{A'}.\]
Therefore $K_{t*}\Phi_{A_t}$ is a path of  l.c.K. structures (for the induced complex structures) with automorphic
 potential and integral period lattice   which connects $\Phi_{A'}$ with $K_{0*}\Phi_A$. Then by Moser stability
 the positive conformal class is the same, which implies that we may scale the eigenvalues at will. If  not all
 eigenvalues have the same norm, they do have the same norm for an obvious Hermitian metric. By connecting that 
metric with the Euclidean one, and using the corresponding path of integral Vaisman structures, we may 
assume without loss of generality
 that all eigenvalues have norm 1/2.

We note that in the previous considerations we really need the path of l.c.K. structures connecting the two sets of
 eigenvalues; the diffeomorphisms that we are considering between different diagonal Hopf manifolds are not holomorphic
 in general, so one cannot apply lemma \ref{lem:lckmoser}.

The final step is correcting the argument of the eigenvalues. If we identify the spheres of radius
 $1$, $1/2$ by the homothety, we obtain the isomorphism
\[(H_{\mathrm{Id}/2},\Phi_{\mathrm{Id}/2},d_{\w_{\mathrm{Id}/2}}^c,\w_{\mathrm{Id/2}})\cong ( S^{2N-1}\times S^1,\Phi_N,\eta_N,d\theta).\]
Also $H_A$ is a mapping torus over $S^1$ with return map $\varphi_A\colon S^{2N-1}\rightarrow S^{2N-1}$, 
which is clearly isotopic to the identity. We can for example construct a path of diagonal matrices $A_t$
 joining $A$ with $\mathrm{Id}/2$  by rotating the eigenvalues clockwise until we reach $1/2$. It is easy 
to produce diffeomorphisms
\[K_t\colon H_{A_t}\rightarrow H_{\mathrm{Id}/2}\] thus getting a path of l.c.K. structures 
 (for the induced complex structure) with automorphic potential and integral period lattice 
 connecting $\Phi_{N}$ with $K_{0*}\Phi_A$.
 \end{proof}

\begin{remark}\label{rem:isot} {\rm  Linear Hopf manifolds are not canonically diffeomorphic to $S^{2N-1}\times S^1$.
 Proposition \ref{pro:linhopf} has produced for $H_A$ many diffeomorphisms $H_A\rightarrow H_{\mathrm{Id}/2}$
 in the same isotopy class taking $\Phi_A$ to the positive conformal class of $\Phi_N$.
Therefore a fortiori the  diffeomorphism class of the positive conformal classes of 
$\Phi_A$ and $\Phi_N$ coincide, and it is in this sense that the positive conformal class 
of $\Phi_A$ in $S^{2N-1}\times S^{1}$ is unique.}
\end{remark}

\begin{proof}[Proof of theorem \ref{thm:mosercor}]

Let $(M,J,\Phi_g)$ be the given l.c.K. structure with automorphic
potential and integral period lattice. Suppose that $\Psi$ is a
holomorphic embedding of $(M,J)$ into the linear Hopf manifold
$(H_A,J_A)$ with a l.c.K. structure $\Phi_A$ with automorphic
potential $r_A$, and Lee form $\w_A$ with integral period lattice
(see \cite{OV10}).

By proposition \ref{pro:linhopf} we have $\phi\colon
H_A\rightarrow S^{2N-1}\times S^1$ a diffeomorphism pulling back
$\Phi_N$ to the positive conformal class of $\Phi_A$. Because $\Psi^*\colon (M,\Psi^*\w_A)\rightarrow  (H_A,\w_A)$ is a morphism,
$\Psi^*\Phi_A$ is $d_{\Psi^*\w_A}$-closed. Non-degeneracy follows from the fact that $\Psi^*\Phi_A(\cdot,J\cdot)$ is the restriction
of the Hermitian metric $\Phi_A(\cdot,J_A\cdot)$. Therefore, by lemma \ref{lem:lckmoser} we just need to show
that $[\Psi^*\w_A]=[\w_g]$ and that $\Psi^*\Phi_A$ is a l.c.K.  structure with automorphic
potential.

To this end we need to recall some aspects of the construction of
 the holomorphic embedding: The manifold $(M,J)$ carries a l.c.K. structure with integral period lattice.
 In the smallest exact covering space  $(\tilde{M},\tilde{J},\Omega)$, whick has deck transformation group $\Gamma\cong \Z$ 
generated
 by the contraction $\gamma$, one constructs  a holomorphic map from the 1-point Stein compactifications
\[\tilde{\Psi}\colon \hat{\tilde{M}}\rightarrow \C^n\]
which is equivariant with respect to the group isomorphism $\Gamma\rightarrow \langle A \rangle$, $\gamma\mapsto A$.
Then one gets the (holomorphic) commutative diagram of morphisms
\begin{equation}\label{cd1}\begin{CD}
(\tilde{M},0) @>\tilde{\Psi}>>(\C^N\backslash \{0\},0) \\
@V\pi_\gamma VV  @V\pi_A VV\\
(M,\Psi^*\w_A) @>\Psi>>(H_A,\w_A)
\end{CD}
\end{equation}
By lemma  \ref{lem:corresp} the additive character $[\w_A]$ characterizes the multiplicative character $\chi_A$. Because
$\tilde{\Psi}$ is equivariant with respect to the action of the groups of deck transformations and the morphism relating
the deck transformation groups is an isomorphism, we have $\Psi^*\chi_A=\chi_{[\w_g]}$. By lemma \ref{lem:corresp} and commutativity
of (\ref{cd1}) we conclude 
\[[\Psi^*\w_A]=[\w_g].\]
Because  $\tilde{\Psi}$ is equivariant w.r.t. deck transformations, it pulls back an automorphic potential for $\Omega_A$ (w.r.t.
 $\chi_{[\w_A]}$) into an automorphic potential for $\tilde{\Psi}^*\Omega_A$ (w.r.t. $\chi_{[\Psi^*\w_A]}$). Therefore 
  by commutativity of   (\ref{cd1}) it follows that 
the conformal class of l.c.K. structures defined by $\Psi^*\Phi_A$ has automorphic potential.

\end{proof}

\begin{remark}{\rm Theorem  \ref{thm:mosercor} holds more generally for arbitrary l.c.K. structures
with automorphic potential (see remark \ref{rem:rank1}).}
\end{remark}

\section{Universal models for reduction of l.c.s.
structures of the first kind}\label{sec:4}

Any symplectic structure in a manifold of finite type can be obtained by reduction of the standard symplectic
structure in $\R^{2n}$ \cite{GT89}. 

Reduction has been extended for l.c.s. structures \cite{HK01}.
Among l.c.s. structures, those of the first kind bear relations
with contact and cosymplectic geometry (see remark
\ref{rem:firstk}). Based on the universal models for reduction for
the latter structures,  one is led to consider the following
family of l.c.s. manifolds of  the first kind: For each pair of
natural numbers $(N,k)$,  we define
\begin{equation}\label{eq:redmodel}
M_{k,N}=\R\times \cJ^1(\mathbb{T}^k\times \R^N), 
\end{equation}
where $\cJ^1(\mathbb{T}^k\times \R^N)$ is the $1$-jet bundle of
the cartesian product of the $k$-dimensional torus and $\R^N$.
Denote by $s$ the coordinate on the first factor in (\ref{eq:redmodel}), by $u$ the real
coordinate in $\cJ^1(\mathbb{T}^k\times \R^N)=\R\times
T^*(\mathbb{T}^k\times \R^N)$, by $(t_1,\dots,t_N)$ the
coordinates in $\R^N$, and by $\theta_1,\dots,\theta_k$ the periodic
coordinates in
 $\mathbb{T}^k$ (with period 1). If   $\mu=(\mu_1,\dots,\mu_k)$ is
 a $k$-tuple
of real numbers, a computation shows that the 1-forms
\[\w_{\mu}=ds+\sum_{j=1}^k\mu_jd\theta_j,\,\a_{k,N}=du-\lambda_{\mathbb{T}^k\times \R^N} \]
fit into a l.c.s. structure of the first kind
\[\Phi_{k,N,\mu}:=d\a_{k,N}-\w_{\mu}\wedge \a_{k,N}.\]
The t.i.a. associated to $\a_{k,N}$  and the anti-Lee vector field
are respectively
\[B=\frac{\partial}{\partial s},\,\;\;\; E=-\frac{\partial}{\partial u}.\]

Let $(M,\Phi,\alpha)$ and $(M',\Phi',\alpha')$  be  l.c.s. manifolds of the first kind. A diffeomorphism
$\Psi$ is said to be of the first kind if it is a strict morphism and satisfies $\Psi^*\a'=\a$.  In such a case, we have that 
$\Psi^*\Phi'=\Phi$, and the associated t.i.a. and the anti-Lee vector fields are $\Psi$-related. 

The action of $SL(k,\Z)$ in $k$-tuples of real numbers $\mu$ is seen to induce an action on
  $(M_{k,N},\Phi_{k,N,\mu},\a_{k,N})$ by diffeomorphisms of 
the first kind. 
 
Before proving that $(M_{k,N},\Phi_{k,N,\mu},\a_{k,N})$ are the
universal models we are looking for, we need to say a few words
about reduction of l.c.s. structures.

Among the results in \cite{HK01},  conditions mimicking coisotropic symplectic reduction 
are imposed on a submanifold $C$ of a l.c.s manifold $(M,\Phi,\w)$,
such that the leaf space associated to the involutive distribution  $\mathrm{ker}\Phi_{\mid C}$ inherits a l.c.s.
structure ($(C,\mathrm{ker}\Phi_{\mid C})$ is a \emph{reductive structure} \cite{HK01}): let $\cF$ be 
the distribution integrating $\mathrm{ker}\Phi_{\mid C}$. Assuming $C/\mathcal{F}$ to be a manifold, one wants the projection
 $(C,\w_{\mid C})\rightarrow C/\mathcal{F}$ to become a strict morphism, so $\w_{\mid C}$ is asked to be $\cF$-basic. Then exactly
the same proof used for symplectic coisotropic reduction produces a l.c.s. form in the quotient whose pullback is 
$\Phi_{\mid C}$. We are interested in finding further conditions so that l.c.s. structures of the first kind are preserved
under reduction.

\begin{lemma}\label{lem:red} Let $(\Phi,\a)$ be a l.c.s. structure of the first kind on $M$ with Lee form 
$\w$ and associated t.i.a. $B$. Let  $C$  be a submanifold of $M$ such that the following properties hold:
\begin{enumerate}
 \item $B$ and $E$ are tangent to $C$.
\item The involutive distribution $\mathrm{ker}{\w}_{\mid C}\cap \mathrm{ker}{\a}_{\mid C}\cap \mathrm{ker}{d\a}_{\mid C}$
 has constant rank, thus defining a foliation $\cF$.
\item The leaf space $M_0=C/\cF$ has a manifold structure induced by the projection $\pi\colon C\rightarrow M_0$.
\end{enumerate}

Then $M_0$ inherits a l.c.s. structure of the first kind $(\Phi_0,\a_0)$ with Lee form $\w_0$ characterized by 
$\pi\colon (C,\w_{\mid C})\rightarrow (M_0,\w_0)$ being a strict morphism such that $\a_{\mid C}=\pi^*\a_0$ 
(and thus $\Phi_{\mid C}=\pi^*\Phi_0$).
The associated t.i.a. and anti-Lee vector fields are the projection of $B$ and $E$ respectively (which are $\cF$-projectable).

We say the $C$ is {\it strongly reducible} and that $(M_0,\Phi_0,\a_0)$ is the reduction of $(M,\a,\w)$ (by $C$).
\end{lemma}
\begin{proof} It is routine to check that $\w_{|C},\a_{|C},d\a_{|C}$ are $\cF$-basic forms,
$B_{|C},E_{|C}$ $\cF$-projectable vector fields and that they
induce in $M_0$ a l.c.s. with the stated properties.
\end{proof}

 We note that
a concatenation of reductions is in an obvious way a reduction in
just one stage. 

\begin{lemma}\label{lm:red0}
If a l.c.s. manifold of the first kind $(M_3,\Phi_3,\alpha_3)$ is the
reduction of a l.c.s. manifold of the first kind
$(M_2,\Phi_2,\alpha_2)$ by the submanifold $C_2\subseteq M_2$ and
if $(M_2,\Phi_2,\alpha_2)$ is the reduction of a l.c.s. manifold
of the first kind $(M_1,\Phi_1,\alpha_1)$ by the submanifold
$C_1\subseteq M_1,$ then $(M_3,\Phi_3,\alpha_3)$ is the reduction
of $(M_1,\Phi_1,\a_1)$ by the submanifold $C_1'=\pi_1^{-1}(C_2)$, where
$\pi_1\colon C_1\to M_2$ denotes the canonical projection.
\end{lemma}

\begin{proof}[Proof of theorem \ref{thm:red}]

Let $(M,\Phi,\a)$ be a l.c.s. structure of the first kind  in a
manifold of finite type. Let  $k$ be its rank and let $\mu\in
\R^k$ be a basis of its period lattice $\Lambda$.

In a first step we construct a l.c.s. manifold of the first kind with the same
 period lattice and whose Lee form has an appearance close to $\w_\mu$, together with a strongly reductive
  submanifold whose reduction is $(M,\Phi,\a)$.

The finiteness of the first Betti number together with the choice of a basis of the period lattice implies that we can write
\[\w=\w_0+\sum_{j=1}^k \mu_j\w_j,\]
where $[\w_j]$, $j=1,\dots,k$, is integral, the classes $\mu_1[\w_1],\dots,\mu_k[\w_k]$ linearly 
independent over the integers, and $\w_0$ is exact.
We fix $\tau_j\colon M\rightarrow S^1$ such that $\tau_j^*d\theta_j=\w_j$ and define

\[M_1=M\times T^*(\mathbb{T}^k).\]
Let $(\theta_j,r_j)$ be the corresponding coordinates on
$\mathbb{T}^k\times \R^k\cong T^*(\mathbb{T}^k).$ Then it can be
checked that the 1-forms
\begin{equation}\label{eq-1}
\a_1=\a+\sum_{j=1}^k\mu_jr_j(d\theta_j-\w_j),\,\;\;\;\w_1=\w_0+\sum_{j=1}^k\mu_jd\theta_j\end{equation}
define  a l.c.s. structure of the first kind $\Phi_1$  with t.i.a. and anti-Lee vector field respectively
\[B_1=B+\sum_{j=1}^k(i_B\w_j)\frac{\partial}{\partial \theta_j},\;\;\;\,E_1=E+\sum_{j=1}^k (i_E\w_j)\frac{\partial}{\partial \theta_j}.\]
We define $C_1$ to be the image of the embedding
\begin{equation}\label{eq-2}
\begin{array}{rcl}
F\colon M\times \R^k &\longrightarrow & M_1\\
         (x,r_j)&\longmapsto & (x,\tau_j(x),r_j).
\end{array}
\end{equation}
A direct computation shows that $C_1$ is a strong reductive
submanifold of $(M_1,\Phi_1,$ $\a_1)$ and that the reduction is
isomorphic to $(M,\a,\w,B)$.

In the second step we construct a l.c.s. manifold of the first kind  with the 
same Lee form and whose potential 1-form has an appearance close to $\a_{N,k}$,
 together with a strongly reductive submanifold whose reduction is $(M_1,\Phi_1,\a_1)$.

 We define
\[M_2=\R\times \cJ^1M_1,\]
with coordinate $s$ for the first factor and $u$ for the real factor of the 1-jet bundle.
The 1-forms
\begin{equation}\label{eq-3}
\w_2=ds+\w_1,\,\;\;\;\a_2=du-\lambda_{M_1}\end{equation}
are seen to define a l.c.s. structure $\Phi_2$ with  t.i.a. and anti-Lee vector field respectively
\[B_2=\frac{\partial}{\partial s}, \;\;\;\;E_2=-\frac{\partial}{\partial u}.\]

As for any l.c.s. structure of the first kind, the vector fields
$B_1,E_1$ have commuting flows. Denote by $C_2$ the open subset of
$\R^2\times M_1$ in which the composition of  both flows is
defined, and embed it in $M_2$ via the map
\begin{eqnarray*}
G\colon C_2&\longrightarrow & M_2\\
         (s,u,x)&\longmapsto & (s,-u,-\a_1(x)).
\end{eqnarray*}
A direct computation shows that $C_2$ is a strongly reductive submanifold
of $(M_2,\Phi_2,$ $\a_2)$ whose reduction is  isomorphic to
$(M_1,\Phi_1,\a_1)$. In fact, if $\phi$ and $\psi$ are the flows of
the vector fields $B_1$ and $E_1$, respectively, then the map
$$C_2\to M_1,\;\;\;(s,u,x_1)\to \psi_u(\phi_s(x_1))$$
induces an isomorphism between the l.c.s. manifold
$(M_1,\Phi_1,\alpha_1)$ and the l.c.s. reduction of
$(M_2,\Phi_2,\alpha_2)$ by the submanifold $C_2.$

In the third step we seek to simplify the formula for the Lee form. To that end we 
will not perform any reduction, just apply an appropriate diffeomorphism of $M_2$.

We take $f_0\colon M\rightarrow \R$ with $\w_0=df_0$. The Lee form can be written
\[\w_2=ds+df_0+\sum_{j=1}^k\mu_jd\theta_j.\]
The diffeomorphism is
\begin{eqnarray*}\label{eq:H}
H\colon M_3:=M_2=\R\times \R\times T^*M_1 &\longrightarrow & M_2\\
         (s,u,\xi,x)&\longmapsto & (s-f_0(x),u,\xi,x).
\end{eqnarray*}
If we pullback the l.c.s. structure we obtain
\[\a_3=\a_2,\,\;\;\w_3=ds+\sum_{j=1}^k\mu_jd\theta_j,\,\;\; B_3=B_2,\,\;\; E_3=E_2.\]

The final step is a further reduction  to make the last simplification of the potential  1-form.
 We take an embedding $M\hookrightarrow\R^{4n}$ which allows us to consider a new embedding $i'\colon M_1=M\times
 T^*\mathbb{T}^k\cong M\times\mathbb{T}^k \times \R^k
 \hookrightarrow  \mathbb{T}^k\times \R^N$, with $N=4n+k.$

 Now, we take the universal l.c.s. manifold  of the first kind

\[(M_{k,N},\Phi_{k,N,\mu},\a_{k,N}).\]

Denote by $\pi\colon M_{k,N}\rightarrow \mathbb{T}^k\times \R^N$
be the bundle map projection. We define
\begin{equation}\label{c4}
C_4=\pi^{-1}(i'(M_1)).
\end{equation}
A final check shows that $C_4$ is a strongly reductive submanifold of
$(M_{k,N},\Phi_{k,N,\mu},$ $\a_{k,N})$ whose reduction is  isomorphic
to $(M_3,\Phi_3,\a_3).$ Thus, using lemma \ref{lm:red0}, we prove
the theorem for the chosen basis $\mu$. 

Different choices of basis are related by the action of $SL(k,\Z)$, which acts on
the corresponding universal manifolds by diffeomorphisms of the first kind. Thus, we rather use $\Lambda$ in 
the notation for our universal manifolds -as in theorem \ref{thm:red}- since it is the $SL(k,\Z)$-orbit what we look at. 
\end{proof}

\section{Universal models for equivariant reduction of l.c.s. structures of the first kind}\label{sec:5}

In this last section we will prove theorem \ref{thm:redeq}, an equivariant version of theorem \ref{thm:red}.
 Let $(M,\Phi,\alpha)$ be a l.c.s. manifold of the first kind with Lee $1$-form $\omega$ and $\psi:G\times M\to M$
 be an action of a Lie group $G$ on $M$. The action is said to be a \emph{l.c.s. action of the first kind}  if each 
automorphism $\psi_g$, $g\in G$, is of the first kind. In such a case, we have that 
 the associated t.i.a. $B$ and the anti-Lee vector field $E$ are $G$-invariant with respect to $\psi.$

Now, we consider $C$ a strong reducible submanifold of $(M,\Phi,\alpha)$ which is $G$-invariant 
with respect to $\psi.$ Denote by $(M_0=C/{\mathcal F},\Phi_0,\alpha_0)$ the reduction 
of $(M,\Phi,\alpha)$ by $C$ (see lemma \ref{lem:red}). Then,  one may easily prove the following result. 

\begin{proposition}\label{prop:action}
Let $\psi\colon G\times M\to M$ be a l.c.s. action of the first kind and $C$ be a $G$-invariant strong
 reducible submanifold of $M$. Then, there exists an  induced l.c.s. action of the first kind 
$\psi_0\colon G\times M_0\to M_0$ of $G$ on the l.c.s. reduced manifold  $(M_0=C/{\mathcal F},\Phi_0,\alpha_0)$ of the first kind.  
\end{proposition}

If the conditions of proposition \ref{prop:action} hold $(M_0,\Phi_0,\alpha_0,\psi_0)$ is said to 
be  {\it  the equivariant reduction of $(M,\Phi,\alpha,\psi)$ by the submanifold $C$.}

Next, we will prove theorem \ref{thm:redeq}. For this purpose, we will use the following lemma (see \cite{MT97}).

\begin{lemma}\label{Lemma-equiv} Let $G$ be a compact and connected Lie group and $\psi:G\times M\to M$ 
be an action of $G$ on a connected manifold $M$. Then, 
\begin{enumerate}
\item If  $\beta$ is  a $k$-form on $M$ with integral cohomology class,  the average
 $\bar{\beta}=\int_G(\psi_g^*\beta) dg$ represents the same integral class, provided $dg$ 
the invariant Haar measure of total volume $1$. 
\item If $f:M\to S^1$ is a smooth map and the $1$-form $\beta=f^*{d\theta}$ is $G$-invariant,
 then there exists a representation $\varphi\colon G\to S^1$ of $G$ on $S^1$ such that $f$ is
 equivariant with respect to usual action of $S^1$ on itself, that is, 
$$f(\psi_g(x))=\varphi(g)\cdot f(x),\;\; \forall g\in G \mbox{ and } x\in M.$$
\end{enumerate}
\end{lemma}

\begin{proof}[Proof of theorem \ref{thm:redeq}.]
In order to prove this theorem, we will rewrite the proof of theorem \ref{thm:red}, adding the corresponding
 equivariant notions. So, 
like in the proof of theorem \ref{thm:red}, we start with  a decomposition of Lee $1$-form $\omega$
 associated with $(\Phi,\alpha)$ 
$$\omega=\omega_0 + \sum_{j=1}^k \mu_j\omega_j.$$
From lemma \ref{Lemma-equiv}, one deduces that $\omega_j$ and the average 
$\bar{\omega}_j=\int_G\psi_g^*(\omega_j)dg$ represent the same integral cohomology class, for $j=1,\dots,k.$
 Therefore, one may suppose without loss of generality, that $\omega_j$ is $G$-invariant. 

Now, we may consider the map $\tau_j\colon M\to S^1$ which satisfies $\tau_j^*d\theta_j=\omega_j,$ for $j=1,\dots,k.$
 Then, using again lemma \ref{Lemma-equiv}, we can choose for each $j$ an action of $G$ on $S^1$
 (induced by a representation $\varphi_j\colon G\to S^1$) such that the map $\tau_j$  is equivariant, i.e., 
$$\tau_j(\psi_g(x))=\varphi_j(g)\cdot \tau_j(x).$$
We  remark that $d\theta_j$ is $G$-invariant with respect to the action $\bar{\psi}_k\colon G\times {\Bbb T}^k\to {\Bbb T}^k$ given by 
$$(\bar{\psi}_k)_g(\theta_1,\dots, \theta_k)=(\varphi_1(g)\cdot \theta_1,\dots ,\varphi_k(g)\cdot \theta_k), $$
with $(\theta_1,\dots ,\theta_k)\in {\Bbb T}^k.$

Next, we introduce the action $\psi_1\colon G\times M_1\to M_1$ of $G$ on the manifold $M_1=M\times T^*({\Bbb T}^k)$ given by 
$$\psi_1(g, (x,\theta,r))=(\psi_g(x),(\bar{\psi}_{k})_g(\theta),r)\;\; \mbox { with } x\in M \mbox{ and }(\theta,r)\in T^*{\Bbb T}^k.$$

Then, we have that the $1$-forms $\alpha_1$ and $\omega_1$  on $M_1$ given in (\ref{eq-1}) are $G$-invariant 
with respect to $\psi_1.$ Moreover, if $F\colon M\times \R^k\to M_1$ is the embedding described in (\ref{eq-2}),
 the submanifold $C_1=F(M\times \R^k)$ is also $G$-invariant. In addition, under the identification of $M$
 with the reduction of $M_1$ by $C_1$,   the induced action from $\psi_1$  on this reduced space  is just $\psi$. 

Take $M_2=\R\times {\mathcal J}^1M_1$ and the cotangent lift of $\psi_1, $ $T^*\psi_1\colon G\times T^*M_1\to T^*M_1,$ 
and we  construct the action $\psi_2\colon G\times M_2\to M_2$ on $M_2$ 
given by $$(\psi_2)_g((s,u,\alpha_{x_1}))=(s,u,T^*(\psi_1)_{g^{-1}}(\alpha_{x_1}))$$
 for  $(s,u,\alpha_{x_1})\in {\Bbb R}\times {\mathcal J}^1M_1$ and $x_1\in M_1.$  Using the fact that 
$\omega_1$ is $G$-invariant with respect $\psi_1$ and $\lambda_{M_1}\in \Omega^1(M_1)$ is $G$-invariant
 with respect  to $T^*\psi_1,$ we deduce that the $1$-forms $\omega_2$ and $\alpha_2$ on $M_2$ described 
in (\ref{eq-3}) are $G$-invariant with respect to $\psi_2.$ 

Note that, since the vector fields $E$ and $B$ and the $1$-form $\alpha_1$ are invariant, it follows that the submanifold 
$C_2$  is $G$-invariant with respect to the action $\psi_2.$  In fact, under the identification of $M_1$ with 
the reduction of $M_2$ by $C_2$, the  induced action by $\psi_2$ is just $\psi_1.$ 

In the third step of the proof of theorem \ref{thm:red}, we have that the map $H\colon M_3=M_2\to M_2$ defined
 in (\ref{eq:H}) is a diffeomorphism. On the other hand,  one may assume that the real function $f_0$
 is also $G$-invariant. It is sufficient to take
$$\widetilde{f_{0}}=\int_G\psi_g^*(f_0)dg$$
which is $G$-invariant and $d\widetilde{f_0}=\omega_0.$
Thus, the diffeomorphism $H\colon M_3\to M_2$ is equivariant and it induces
 a new action $\psi_3\colon G\times M_3\to M_3$  such that $\alpha_3,\omega_3, B_3$ and $E_3$ are $G$ invariant. 

Since $G$ is compact and $M$ is of finite type, from the Mostow-Palais theorem \cite{Mo,Pa}, we deduce
 that there exist an integer $n$, an orthogonal action of $G$ on ${\Bbb R}^n$ and an equivariant
 embedding $i\colon M\hookrightarrow {\Bbb R}^n$. Therefore, we have an orthogonal action of $G$ 
on ${\Bbb T}^k\times {\Bbb R}^N$ with $N=n+k$ 
$$\bar{\psi}_{k,N}\colon G\times {\Bbb T}^k\times {\Bbb R}^N\to {\Bbb T}^k\times {\Bbb R}^N$$ given by 
$$(\bar\psi_{k,N})_g(z,r,a)=((\bar\psi_{k})_g(z), r, g\cdot a),$$
with  $(z,r,a)\in {\Bbb T}^k\times {\Bbb R}^k\times {\Bbb R}^n.$  
Thus, we may consider the l.c.s. action on $(M_{k,N},\Phi_{k,N,\mu},$ $\alpha_{k,N},\omega_{\mu})$ defined by 
$$(\psi_{k,N})_g(s,u,\gamma_{(z,t)})=(s,u,T^*(\bar{\psi}_{k,N})_{g^{-1}}(\gamma_{(z,t)}))$$
for $(s,u,\gamma_{z,t})\in M_{k,N}$ and $(z,t)\in {\Bbb T}^k\times {\Bbb R}^N.$
Note that the $1$-form $\alpha_{k,N}=du-\lambda_{{\Bbb T}^k\times {\Bbb R}^N}$ is $G$-invariant 
with respect to $\psi_{k,N}$. Moreover, since $d\theta_j$ is $G$-invariant with respect to 
$\bar{\psi}_k$ then $\omega_{\mu}$ is $G$-invariant with respect to $\psi_{k,N}.$ 

The induced embedding $i'\colon M_1=M\times {\Bbb T}^k\times {\Bbb R}^k\to {\Bbb T}^k\times {\Bbb R}^N$ by 
$i\colon M\hookrightarrow {\Bbb R}^n$, is $G$-invariant with respect to $\psi_1$ and $\bar\psi_{k,N}.$
 Thus, $i'(M_1)$ is $G$-invariant with respect to $\bar\psi_{k,N}$. Since the projection
 $\pi\colon M_{k,N}\to {\Bbb T}^k\times {\Bbb R}^N$ is $G$-invariant with respect to 
$\psi_{k,N}$ and $\bar{\psi}_{k,N}$ we conclude that $C_4=\pi^{-1}(i'(M_1))$
 is $G$-invariant with respect to $\psi_{k,N}.$  Finally, under the identification 
of $M_3$ with the reduction of $M_{k,N}$ by $C_4$, the induced action from $\psi_{k,N}$ is just $\psi_3. $

The action of $SL(k,\Z)$ is equivariant w.r.t. to the action of $\bar{\psi}_{k,N}$, so may consider the $SL(k,\Z)$-orbit
and cut down the dependence of the construction from the basis $\mu$ to the lattice $\Lambda$.
\end{proof}
\begin{remark}
 {\rm It is natural to  try to define universal models for Vaisman manifolds via reduction. One obstacle we find is that
our universal l.c.s. manifolds of the first kind  $(M_{k,N},\Phi_{k,N,\mu},$ $\a_{k,N})$ do not seem to admit compatible
Vaisman structures in a straightforward manner. As for the process of reduction itself, l.c.K. coisotropic reduction 
can be defined
in an obvious way: to define l.c.s. reduction one requires a regular foliation with smooth leaf space which
 integrates
$\mathrm{ker}\Phi_{\mid C}$, and requires $\w_{\mid C}$ to be $\cF$-basic. The additional ingredient is an integrable compatible
 almost complex
structure in the leaf space. It is reasonable then to further ask (1) $TC\cap JTC$ to be of constant rank and complementary to 
$\cF$ in $C$, and  (2) the CR structure $(C, TC\cap JTC)$ to be $\cF$-basic, i.e. invariant by flows of vector fields
 tangent to $\cF$.
Of course, what is difficult is to give geometric conditions which imply that l.c.K. coisotropic reduction is possible. 
This was done
in \cite{GOP} for twisted Hamiltonian actions by automorphisms of the structure (preserving $J$ and the conformal class of $\Phi$).
For Vaisman coisotropic reduction, one further adds the requirement of $C$ being stable under the holomorphic flow of $B-iJE$. If
a Vaisman manifold is acted upon by a group of Vaisman automorphisms, then the action  by definition  commutes  with the flow of 
 $B-iJE$, it is twisted  Hamiltonian, and free on the inverse image of zero if this is non-empty \cite{GOP}, so coisotropic Vaisman 
reduction is possible.}
\end{remark}

\section{Conclusions and future work}
Universal models for several types of l.c.s. manifolds associated
with embedding or reduction procedures are obtained. The existence
of these  universal models for embeddings (in the compact case) is related with the search of a universal
model for a compact manifold $M$ endowed with an arbitrary $1$-form $\Theta.$
In this case one may embed the manifold into a sphere $S^{2N-1}$ and the   pullback of the
 standard contact $1$-form on $S^{2N-1}$ is just $\Theta$ (up to the multiplication by a positive constant).
In relation with previous
results, our method allows to cut down  substantially the
dimension of the sphere. In the particular case of a compact contact
manifold $M$, we give a simple proof about how to obtain a contact
embedding (up to the multiplication by a positive constant) from $M$ to
$S^{2N-1}.$

Using these results, we have seen that the universal model (via embeddings) of a
compact exact  l.c.s. manifold with integral period lattice is
the cartesian product $S^{2N-1}\times S^1$ with the standard
l.c.s. structure. In the particular case of a l.c.K. structure
with automorphic potential and integral period lattice on a
compact manifold $M$, we have discussed  the relation between the l.c.s.
embedding of $M$ into $S^{2N-1}\times S^1$ and recent holomorphic embedding results
for this type of manifolds.

Finally, we have obtained that a universal model for a l.c.s.
manifold (of finite type) of the first kind via a reduction
procedure is ${\mathbb R}\times {\mathcal J}^1({\mathbb T}^k\times {\mathbb
R}^N)$  endowed with a suitable l.c.s. structure. An equivariant version of this result has been presented
 at the end of the paper. 

It would be interesting to  pursue the existence of universal models  (for embedding and reduction procedures)
for  arbitrary l.c.s. manifolds.

\appendix
\section{Non-exactness of the Oeljeklaus-Toma l.c.K. structures}\label{appendix}
In this appendix we will show that the Oeljeklaus-Toma l.c.K. manifolds are not exact. 

We briefly recall the construction of the  Oeljeklaus-Toma l.c.K. structures (for details see \cite{OT05,PV} and references
therein):

Let $K$ be an algebraic number field of degree $n$ and let $\sigma_1,\dots,\sigma_n$ be the distinct embeddings of $K$
into $\mathbb{C}$. Assume that $\sigma_1,\dots,\sigma_{n-2}$ are real and  $\sigma_{n-1}$ and $\sigma_n$ are non-real.
 Let $\mathcal{O}_K$
denote the ring of algebraic integers of $K$, which is a rank $n$ free $\mathbb{Z}$-module. Let $\mathcal{O}_K^{*,+}$ denote
the positive units, i.e. those units $u\in \mathcal{O}_K^*$ such that 
\[\sigma_i(u)>0,\,i=1,\dots,n-2.\]
According to  Oeljeklaus and Toma the actions 
\[T_a(z_1,\dots,z_{n-1}):=(z_1+\sigma_1(a),\dots,z_{n-1}+\sigma_{n-1}(a)), \, a\in  \mathcal{O}_K,\]
\[R_u(z_1,\dots,z_{n-1}):=(\sigma_1(u)z_1,\dots,\sigma_{n-1}(u)z_{n-1}), \,u\in  \mathcal{O}_K^{*,+}\]
fit into a free co-compact action of the semi-direct product $\mathcal{O}_K\rtimes \mathcal{O}_K^{*,+}$ on
$\mathbb{H}^{n-2}\times \C$, where $\mathbb{H}$ denotes the upper half plane.
The corresponding quotient 
\[(M_K,J_K):=\mathbb{H}^{n-2}\times \C/\mathcal{O}_K\rtimes \mathcal{O}_K^{*,+}\]
is called an Oeljeklaus-Toma manifold.

Consider the function 
\begin{eqnarray}\label{eq:autfunction}\nonumber
 r\colon \mathbb{H}^{n-2}&\longrightarrow &\mathbb{R}\\
(z_1,\dots,z_{n-2})&\longmapsto & \prod_{i=1}^{n-2}(\mathrm{im}z_i)^{-1}
\end{eqnarray}
and the standard 2-form
\[\Phi_{\mathrm{std}}=dz_{n-1}\wedge d\bar{z}_{n-1}\in \Omega^{1,1}(\C),\]
and define
\[\Omega=\partial\bar{\partial}\phi +\Phi_{\mathrm{std}}.\]
Then $\Omega$ is a K\"ahler form on $\mathbb{H}^{n-2}\times \C$ such that 
\begin{equation}\label{eq:otchar}
T_a^*\Omega=\Omega,\,a\in \mathcal{O}_K,
\end{equation}
and 
\[
R_u^*\Omega =|\sigma_{n-1}(u)|^2\Omega,\,u\in \mathcal{O}_K^{*,+}.
\]
Hence $\mathcal{O}_K\rtimes \mathcal{O}_K^{*,+}$ acts by K\"ahler homotheties giving rise to a multiplicative character $\chi$,
and therefore the K\"ahler form descends to 
a conformal class of l.c.K. structures with Lee class associated to $\chi$ (see lemma \ref{lem:corresp}). We let $\Phi_K$ be a representative of the 
induced conformal class of l.c.K. structures.

Note that $r$ in (\ref{eq:autfunction}) is an automorphic function, but the function $r+z_{n-1}\bar{z}_{n-1}$
 -which is a $dd^c$-potential for $\Omega$-
is not automorphic. That no automorphic potential for $\Omega$ can exist is a consequence of the following result: 
\begin{proposition}\label{pro:otnonexact} The Oeljeklaus-Toma l.c.K. manifold $(M_K,J_K, \Phi_K)$ is non-exact.
\end{proposition}
\begin{proof} By lemma \ref{lem:corresp} exactness of $(M_K,J_K,\Phi_K)$ is equivalent to  
\begin{equation}\label{eq:otexact}
 \Omega=d\alpha, \, \alpha\in \Omega^1(\mathbb{H}^{n-2}\times \C)^\chi.
\end{equation}
Because $r$ is automorphic   (\ref{eq:otexact}) is equivalent to 
\begin{equation}\label{eq:otexact2}
\Phi_{\mathrm{std}}=d\alpha, \, \alpha\in \Omega^1(\mathbb{H}^{n-2}\times \C)^\chi.
\end{equation}
Let us assume that (\ref{eq:otexact2}) holds. 

Let us write $\mathbb{H}=\mathbb{R}\times\mathbb{R}^{>0}$  and
\[
\mathbb{H}^{n-2}\times \C={(\mathbb{R}^{>0})}^{n-2}\times \mathbb{R}^{n-2}\times \mathbb{C}.
\]
Because (i) the action by translations of $\mathcal{O}_K$ on $\mathbb{H}^{n-2}\times \C$ is trivial on the factor
${(\mathbb{R}^{>0})}^{n-2}$ ($\sigma_i(a)\in \mathbb{R}$, $a\in \mathcal{O}_K$, $i=1,\dots,n-2$) and (ii)
 $\sigma(\mathcal{O}_K)\subset 
 \mathbb{R}^{n-2}\times \mathbb{C}$
is a lattice of full rank \cite{OT05}, we have 
\[\mathbb{H}^{n-2}\times \C/\mathcal{O}_K\cong {(\mathbb{R}^{>0})}^{n-2}\times {\mathbb{T}}^{n}.\]
According to (\ref{eq:otchar}) the restriction of $\chi$ to $\mathcal{O}_K$ is trivial and thus
both $\alpha$ and $\Phi_{\mathrm{std}}$ descend to forms $\hat{\alpha}, \hat{\Phi}_{\mathrm{std}}$ on 
${(\mathbb{R}^{>0})}^{n-2}\times {\mathbb{T}}^{n}$. By (\ref{eq:otexact2}) we
obtain
\begin{equation}\label{eq:otexact3}
\hat{\Phi}_{\mathrm{std}}=d\hat{\alpha}.
\end{equation}
Note that because $\Phi_{\mathrm{std}}$ is constant, it is invariant by any translation in 
${(\mathbb{R}^{>0})}^{n-2}\times \mathbb{R}^{n-2}\times \mathbb{C}$. In particular 
$\hat{\Phi}_{\mathrm{std}}\in \Omega^2({(\mathbb{R}^{>0})}^{n-2}\times {\mathbb{T}}^{n})$ is invariant by the $\mathbb{T}^{n}$-action.
Fix a Haar measure in  $\mathbb{T}^{n}$ of total volume 1 and denote the average of $\hat{\alpha}$ by $\int\hat{\alpha}$.
Average (\ref{eq:otexact3}) and use the invariance of $\hat{\Phi}_{\mathrm{std}}$ to get
\begin{equation}\label{eq:otexactaver}
\hat{\Phi}_{\mathrm{std}}=d\int \hat{\alpha}.
\end{equation}
Fix any $\{y\}\in {(\mathbb{R}^{>0})}^{n-2}$ and the corresponding torus $\mathbb{T}^{n}:=\{y\}\times \mathbb{T}^{n}$.
 The result of restricting (\ref{eq:otexactaver})
to this torus is  
\begin{equation}\label{eq:otexactaverres}
\hat{\Phi}_{\mathrm{std}\mid \mathbb{T}^{n}}=d(\int \hat{\alpha}_{\mid \mathbb{T}^{n}}).
\end{equation}
 Observe that by construction $\hat{\Phi}_{\mathrm{std}\mid \mathbb{T}^{n}}$ is a non-trivial 2-form. On the other hand
 the restriction $\int \hat{\alpha}_{\mid \mathbb{T}^{n}}$ is an invariant 1-form, and thus its exterior differential must vanish, which
contradicts (\ref{eq:otexactaverres}).

\end{proof}


\begin{thebibliography}{xxxxx}

\bibitem{BK09} Bande, G.; Kotschick, D.: Moser stability for locally conformally symplectic structures.  Proc. Amer. Math. Soc.  137  (2009),  no. 7, 2419--2424.

\bibitem{BK10} Bande, G.; Kotschick, D.: Contact pairs and locally conformally symplectic structures. Preprint arXiv:1006.0315 (2010).

\bibitem{Ba02} Banyaga, A.: Some properties of locally conformal symplectic structures.  Comment. Math. Helv.  77  (2002),  no. 2, 383--398.

\bibitem{Ba08} Banyaga, A.: Examples of non $d_\omega$-exact locally conformal symplectic forms.  J. Geom.  87  (2007),  no. 1-2,  1--13.

\bibitem{Bo} Bourgeois, F.:  Odd dimensional tori are contact manifolds.  Int. Math. Res. Not.  2002,  no. 30, 1571--1574.


\bibitem{DLM} Dazord, P.;  Lichnerowicz, A.; Marle,  Ch. M.: Structure locale des vari\`et\'es de Jacobi. J. Math. Pures Appl. 70 (1991), 101--152.

\bibitem{FF10} Fernandes, R.; Frejlich  P.: A h-principle for symplectic foliations. Preprint arXiv:1010.3447 (2010). 

\bibitem{GO98} Gauduchon, P.; Ornea, L.: Locally conformally K\"ahler metrics on Hopf surfaces.  Ann. Inst. Fourier (Grenoble)  48  (1998),  no. 4, 1107--1127.

\bibitem{GOP} Gini, R.; Ornea, L.; Parton, M.:  Locally conformal K\"ahler reduction.  J. Reine Angew. Math.  581  (2005), 1--21.

\bibitem{GT89} Gotay, M. J.; Tuynman, G. M.: $R^{2n}$ is a universal symplectic manifold for reduction.  Lett. Math. Phys.  18  (1989),  no. 1, 55--59.

\bibitem{GL} Gu\'edira, F.; Lichnerowicz, A.: G\'eom\'etries des alg\'ebres  de Lie locales de Kirillov. J. Math. Pures Appl. 63 (1984), 407--484.


\bibitem{H} Haller, S.: Perfectness and Simplicity of Certain Groups of Diffeomorphisms. Thesis, University of Vienna, 1998.


\bibitem{HK99} Haller, S.; Rybicki, T.: On the group of diffeomorphisms preserving a locally conformal symplectic structure.  Ann. Global Anal. Geom.  17  (1999),  no. 5, 475--502.

\bibitem{HK01} Haller, S.;  Rybicki, T.: Reduction for locally conformal symplectic manifolds.  J. Geom. Phys.  37  (2001),  no. 3, 262--271.

\bibitem{KO05} Kamishima, Y.; Ornea, L.: Geometric flow on compact locally conformally K\"ahler manifolds.  Tohoku Math. J. (2)  57  (2005),  no. 2, 201--221.

\bibitem{Ki} Kirillov, A.A.: Local Lie algebras, Russ. Math. Surv. 31 (1976), 55--75.


\bibitem{LeLoMaPa} de Le\'on, M.; L\'opez, B.;  Marrero, J.C.; Padr\'on, E.: On the computation of the
Lichnerowicz-Jacobi cohomology.  J. Geom. Phys. 44 (2003), 507--522.

\bibitem{MT97} de Le\'on, M.; Tuynman, G: M.:  A universal model for cosymplectic manifolds.  J. Geom. Phys.  20  (1996),  no. 1, 77--86.

\bibitem{LMT98} de Le\'on, M.; Marrero, J. C.; Tuynman, G. M.: $R^{2n+1}$ is a universal contact manifold for reduction.  J. Phys. A  30  (1997),  no. 5, 1605--1612.


\bibitem{Ma} Marle, Ch.M. : A property of conformally Hamiltonian vector fields; Application to the Kepler problem, Preprint arXiv:1011.5731 (2010).

\bibitem{MMOPR} Marsden, J.E.; Misiolek, G.; Ortega, J.P.; Perlmutter, M.;  Ratiu, T.S.: Hamiltonian Reduction by Stages. Lecture Notes in Mathematics, volume 1913, Springer Verlag.   (2007).

\bibitem{Mo} Mostow, G.D.: Equivariant embedding in Euclidean space. Ann. Math. 65 (1957) 432--446. 

\bibitem{NR61} Narasimhan, M. S.; Ramanan, S.: Existence of universal connections.  Amer. J. Math.  83  (1961) 563--572.

\bibitem{OT05} Oeljeklaus, K.; Toma, M.: Non-K\"ahler compact complex manifolds associated to number fields.  Ann. Inst. Fourier (Grenoble)  55  (2005),  no. 1, 161--171.

\bibitem{OV03} Ornea, L.; Verbitsky, M.: Structure theorem for compact Vaisman manifolds.  Math. Res. Lett.  10  (2003),  no. 5-6, 799--805.

\bibitem{OV09} Ornea, L.; Verbitsky, M.: Morse-Novikov cohomology of locally conformally K\"ahler manifolds.  J. Geom. Phys.  59  (2009),  no. 3, 295--305.

\bibitem{OV10}  Ornea, L.; Verbitsky, M.: Locally conformal K\"ahler manifolds with potential. Math. Ann. 248 no. 1 (2010),  25--33.

\bibitem{OV10a} Ornea, L.; Verbitsky, M.: Topology of locally conformally K\"ahler manifolds with potential.  Int. Math. Res. Not. IMRN  (2010)  no. 4, 717--726.


\bibitem{OV10b} Ornea, L.; Verbitsky, M.: Locally conformally K\"ahler manifolds admitting a holomorphic conformal flow. Preprint arXiv:1004.4645 (2010). 

\bibitem{OR} Ortega, J.P.;  Ratiu, T. S.: Momentum Maps and Hamiltonian Reduction. Progress in Mathematics, volume 222. Birkh\"auser Boston, Inc., Boston, MA. xxxiv+497 pp. ISBN: 0-8176-4307-9 (2004). 

\bibitem{Pa} Palais, R.S.: Embedding of compact differentiable transformation groups in orthogonal presentations. J. Math. Mech. 6, (1957) 673--678. 

\bibitem{PV} Parton, M.; Vutulescu, V.: Examples of non-trivial rank in locally conformal K\"ahler geometry.   To appear in Math. Z. DOI 10.1007/s00209-010-0791-5, arXiv:1001.4891 (2010). 

\bibitem{Ti} Tischler, D.: Closed $2$-form and an embedding theorem for symplectic manifolds. J. Diff. Geometry, 12 (1977) 229--235.

\bibitem{Va85} Vaisman, I.: Locally conformal symplectic manifolds.  Internat. J. Math. Math. Sci.  8  (1985),  no. 3, 521--536.







\end{thebibliography}
\end{document}